\documentclass[reqno,12pt]{amsart}
\usepackage{amssymb,amsmath,amsthm,amstext,amsfonts,color}
\usepackage[dvips]{graphicx}
\usepackage[ansinew]{inputenc}
\usepackage{a4wide}
\usepackage{psfrag}

\newcommand{\qand}{\quad\text{and}\quad}

\newcommand{\const}{{\rm const\,}}

\theoremstyle{plain}
\newtheorem{maintheorem}{Theorem}

\newtheorem{maincorollary}{Corollary}

\newtheorem{theorem}{Theorem }[section]
\newtheorem{proposition}[theorem]{Proposition}
\newtheorem{lemma}[theorem]{Lemma}

\newtheorem{question}{Question}
\newtheorem{conjecture}{Conjecture}

\theoremstyle{definition} \theoremstyle{remark}
\newtheorem{remark}[theorem]{Remark}

\newtheorem{definition}[theorem]{Definition}

\newcommand{\field}[1]{\mathbb{#1}}

\newcommand{\RR}{\field{R}}

\newcommand{\ZZ}{\field{Z}}

\newcommand{\diam}{\operatorname{diam}}
\newcommand{\dist}{\operatorname{dist}}
\newcommand{\supp}{\operatorname{supp}}

\newcommand{\sgn}{\operatorname{sgn}}

       \newcommand{\De}{\Delta}

\renewcommand{\epsilon}{\varepsilon}

\newcommand{\la} {\lambda}

\newcommand{\si} {\sigma}

\newcommand{\vfi}{\varphi}

\newcommand{\cH}{{\mathcal H}}

\newcommand{\cP}{\mathcal{P}}
\newcommand{\cL}{\mathcal{L}}

\newcommand{\cQ}{\mathcal{Q}}
\newcommand{\cF}{\mathcal{F}}

\newcommand{\cU}{\mathcal{U}}

\def \fX {{\mathfrak X}}

\newcommand{\Leb}{\mbox{Leb}}

\newcommand{\ov}{\overline}
\newcommand{\R}{{\mathbb R}}

\begin{document}

\title[Robust exponential decay of correlations]
{Robust exponential decay of correlations for singular-flows}

\author{V\'itor Ara\'ujo and Paulo Varandas} 

\address{V\'itor Ara\'ujo, Departamento de Matem\'atica,
  Universidade Federal da Bahia\\
  Av. Ademar de Barros s/n, 40170-110 Salvador, Brazil.}

\email{vitor.araujo@pq.cnpq.br}

\address{Paulo Varandas, Departamento de Matem\'atica, 
  Universidade Federal da Bahia\\
  Av. Ademar de Barros s/n, 40170-110 Salvador, Brazil.}

\email{paulo.varandas@ufba.br}

\thanks{Part of this work was done during P.V. postdoctoral
  period at UFRJ-Rio de Janeiro with the finantial support
  of FAPERJ (Brazil-Rio de Janeiro).  P.V was partially
  supported by FAPESB. V.A. was partially supported by CNPq,
  PRONEX-Dyn.Syst.  and FAPERJ}

\date{\today}

\begin{abstract}
  We construct open sets of $C^k$ ($k\ge 2$) vector fields
  with singularities that have \emph{robust exponential
    decay of correlations} with respect to the unique
  physical measure.  In particular we prove that the
  geometric Lorenz attractor has exponential decay of
  correlations with respect to the unique physical measure.
\end{abstract}

\keywords{robust exponential decay of correlations,
  geometric Lorenz flow, uniform non integrability, induced
  uniformly expanding Markov map, hyperbolic skew-product
  semiflow, smooth stable foliation}

\subjclass{
37D30, 37A25,
37C10
}

\maketitle

\tableofcontents

\section{Introduction}

The ongoing interest in the mixing properties of deterministic
dynamical systems was strongly inspired by the relevance of the
subject in statistical mechanics. Moreover, the mixing properties
of any given equilibrium state usually require deep knowledge
of the system's chaotic features given for instance by its Lyapunov exponents. 

Thermodynamical formalism was brought into the realm of dynamical 
systems by works of Sinai, Ruelle and  Bowen~\cite{Si72,Bo75,Ru76}. In fact,
since uniformly hyperbolic maps are semi-conjugated to subshifts of finite type
(via Markov partitions), any transitive uniformly hyperbolic map has a unique
equilibrium state $\mu_\phi$ for every H\"older continuous potential
$\phi$ with good mixing conditions:  the correlation function
\begin{align}\label{eq:defcorrfunction}
  C_n(\vfi,\psi) = \Big|\int\vfi\cdot\big(\psi\circ f^n\big)
  \, d\mu_\phi -\int\vfi\,d\mu_\phi\int \psi\, d\mu_\phi\Big|
\end{align}
decays exponentially fast among H\"older continuous observables. 
Roughly, the robust chaotic features of uniformly hyperbolic dynamics  are
responsible by the very strong mixing properties.
An extension of the thermodynamical formalism for uniformly
hyperbolic (Axiom A) flows was also possible by the construction of finite 
Markov partitions obtained by Bowen and Ruelle in~\cite{BR75}. 


Even though uniformly hyperbolic (Axiom A) flows are
semi-conjugated to suspension semiflows over subshifts of
finite type, the mixing properties for time-continuous
dynamical systems turned out to be much more subtle than the
discrete time setting. On the one hand, hyperbolic
suspension flows by a constant roof function are never
topologically mixing despite the exponential mixing rate of
the base transformation.  {On the other
  hand, the general feeling that topologically mixing
  uniformly hyperbolic flows should enjoy exponential decay
  of correlations was promptly put down by the
  counterexample of Ruelle~\cite{ruelle1983}, showing that
  there are topologically mixing Axiom A flows without
  exponential decay of correlations. In fact, examples of
  hyperbolic flows with arbitrary slow rate of decay of
  correlations were given by Pollicott~\cite{Pol85}.}
Hence, despite several recent contributions, a complete
understanding of the mixing properties for uniformly
hyperbolic flows is still far from complete.

For a more detailed description of the state of the art let
us mention that the proof of exponential decay of
correlations for geodesic flows on manifolds of constant
negative curvature was first obtained in two dimensions by
Collet-Epstein-Gallavotti~\cite{CEG84} and then in three
dimensions by Pollicott \cite{pollicott92} through group
theoretical arguments.

Much more recently, Chernov~\cite{chernov98} and
Dolgopyat~\cite{Do98} studied Anosov flows and
Liverani~\cite{liverani2004} extended such results for
contact flows. Still in the uniformly hyperbolic context,
Pollicott~\cite{pollicott99} extended the results in
\cite{Do98} to study the decay rate of equilibrium states
associated to H\"older continuous potentials and Anosov
flows.  {We should also refer that Dolgopyat
  \cite{dolgopyat98} proved that typical Axiom A flows (in a
  probabilistic sense) have superpolynomial decay of
  correlations.  In a more recent contribution, Field,
  Melbourne and T\"orok~\cite{FMT} proved that in fact a
  $C^2$-open and $C^\infty$-dense set of Axiom A flows have
  superpolynomial decay of correlations. Hence, Axiom A
  flows have robust fast mixing.} This raises the following
natural question.

  \begin{question}
    \label{ques:robust-exp-decai-AxiomA}
    Is there some hyperbolic (non-singular) flow with robust
    exponential decay of correlations?
  \end{question}

%
%

  Rather surprisingly, there are some recent evidences that
  the presence of singularities appears as one important
  mechanism to obtain topologically and measure theoretical
  mixing. Roughly, orbits that approach singularities
  accelerate differently causing displacements of different
  order in the flow direction.

  One of the most emblematic examples of flows where
  singular and regular behavior coexist are Lorenz
  attractors. {In the mid seventies
  Afraimovich, Bykov and Shil'nikov \cite{ABS77} and
  Guckenheimer and Williams \cite{GW79}
  introduced independently the geometric Lorenz attractors 
  to model the original Lorenz attractors.}
  The rigorous proof for the existence of the
  non-uniform hyperbolic Lorenz attractor for the parameter
  values originally suggested by E. Lorenz was obtained by
  Tucker~\cite{Tu99} in the end of the 1990's. In addition,
  there was a general feeling that the Lorenz attractor
  should have a unique physical measure with exponential
  decay of correlations.

  In more recent works there were several advances in that
  direction. First Luzzatto, Melbourne, Paccaut \cite{LMP05}
  proved that the geometric models for the Lorenz attractor
  are topologically mixing with respect to the unique
  physical measure. Then Ara\'ujo, Pac\'ifico, Pujals,
  Viana~\cite{APPV} guarantee that every singular hyperbolic
  attractor (a class which contains the Lorenz and
  geometrical Lorenz attractors) carries a unique physical
  measure whose basin of attraction covers Lebesgue almost
  every point. { Recently,
    in~\cite{galapacif09,ArGalPac}, exponential decay of
    correlation for the Poincaré return map to suitably
    chosen cross-sections of geometric Lorenz flows
    and for the general case of singular-hyperbolic flows
    has been obtained.}  However the question about the
  exponential decay rate of the flow on this class of
  attractors remained open.


  More recent developments include a criterium given by
  Baladi, Vall\'ee \cite{BaVal2005} and \'Avila, Go\"uezel,
  Yoccoz \cite{AvGoYoc} to deduce exponential decay of
  correlations for suspension flows over $C^2$ Markov maps.
  The fact that here Markov partition may admit countably
  many elements implies that those classes of systems
  contain many important examples in non-uniformly
  hyperbolic dynamics as the suspension semiflows of the
  Maneville-Pommeau map, the H\'enon maps, and other classes
  of flows as 
  singular-hyperbolic attractors
  (e.g. the Lorenz or geometrical Lorenz
  attractor). {A similar approach was pursued
    in~\cite{bufetov06} by Bufetov to obtain streched
    exponential decay of correlations for the Teichm\"uller
    flow on the space of Abelian differentials.}

  Our purpose here is to contribute to the ergodic theory of
  singular flows and to construct a nonempty open subset of singular
  flows with exponential decay of correlations. Let us
  mention that {Question~1} is not answered in the
  uniformly hyperbolic context. Such class of flows,
  including the geometric Lorenz attractor, combine
  hyperbolic behavior with the existence of singularities.

Our strategy is to construct Lorenz attractors whose
associated one-dimensional piecewise expanding
transformation is twice differentiable and, hence, admits a
countable Markov partition as in the above setting. Then, we prove
that these flows are conjugated to suspension flows over a
base with an hyperbolic structure and such that the height
function satisfies a \emph{uniform non-integrability}
condition as introduced by Dolgopyat. Then we use a criteria
from \cite{BaVal2005,AvGoYoc} to deduce that such flows have
exponential decay of correlations. Moreover, using
\cite{MeTo04,HoMel} we are also able to prove that these
flows satisfy the central limit theorem.

Two final comments are in order. First let us mention that
the Lorenz attractors associated to the original parameters
obtained by E. Lorenz in \cite{Lo63} do not verify our
assumption, and so the question of exponential decay of
correlations for the original Lorenz attractor and small
perturbations of it remains open. The second comment is that
we expect that a similar approach may be applied
to other equilibrium states.


\begin{question}\label{quest1}
  Do the equilibrium states constructed by Pac\'ifico and
  Todd~\cite{PT09} for the contracting Lorenz attractor have
  exponential decay of correlations?
\end{question}

We finish with the following  conjectures on the decay of correlations for
general robustly transitive flows.  

\begin{conjecture}\label{conj1}
  Non-hyperbolic robustly {mixing} flows in
  three-dimensional manifolds have robust exponential decay
  of correlations.
\end{conjecture}

It is known that robustly transitive flows in dimension
three are singular-hyperbolic, that is, are partially
hyperbolic with one-dimensional contracting direction and a
two-dimensional volume expanding direction; see
e.g. \cite{MPP04}. Moreover, if there are no singularities
then the flow is uniformly hyperbolic and the decay of
correlations of the SRB measure can be arbitrarily slow.
See e.g. \cite{AraPac2010} for a rather complete description
of the state of the art. {However, it is not yet known 
wether all singular-hyperbolic flows are topologically mixing.}
Therefore the previous conjecture states that singularities are a 
mechanism to generate robust
exponential decay of correlations in dimension three. Indeed
we believe that it is possible to remove the regularity
assumption.

\begin{conjecture}\label{conj2}
  $C^{1+\alpha}$ Lorenz attractors have robust exponential
  decay of correlations.
\end{conjecture}

The reduction to a suspension semiflow over a non-uniformly
expanding base transformation can also be performed in a
higher dimensional class of examples; see
\cite{BPV97}. These are just a particular example of a
sectional-hyperbolic attractor; see \cite{MeMor06}. The
notion of sectional-hyperbolicity generalizes the notion of
hyperbolicity for singular flows in any dimension and
contains, in particular, the class of singular hyperbolic
attractors in $3$-manifolds.

\begin{conjecture}
  \label{conj3}
  Smooth singular flows in higher dimensions which are
  sectionally-hyperbolic exhibit robust exponential decay of
  correlations.
\end{conjecture}

We believe our main result and its proof can be adapted to
establish limit theorems for the distribution of random
variables generated by the geometric Lorenz system.  On the
one hand, in a recent work, Holland and
Melbourne~\cite{HoMel} used that the geometric Lorenz
attractor is a suspension flow to prove that all Lorenz
attractors satisfy the central limit theorem and invariance
principles, where no condition on the speed of decay of
correlations was necessary. On the other
hand, limit theorems for diffeomorphisms given as time-$t$
maps of flows are harder to obtain.  Notice that even for
uniformly hyperbolic flows the time-one maps are partially
hyperbolic diffeomorphisms. A very general result was
obtained by Melbourne and T\"orok~\cite{MelTor02} under some
assumptions on the decay of correlations for the flow.  We
expect these ideas can be adapted to prove that the strong
mixing properties for the $C^2$-open subset of geometric
Lorenz attractors imply (robust) limit theorems for the
corresponding time-one maps. More precisely we pose the
following:

\begin{conjecture}\label{conj:CLT-Time1}
  Let $\cU\subset \fX^s(M)$ be the open family of vector
  fields for which we prove exponential decay of
  correlations, and denote by $(X_t)_t$ the flow generated
  by $X\in \cU$. For all but countably many values of
  $t\in\mathbb R$ the time-$t$ map $X_t$ the following
  Central Limit Theorem holds: for any $\varphi: \De_r \to
  \R$ in $L^\infty(\De_r)$ there exists
  $\sigma=\sigma(\varphi)>0$ such that
\begin{equation*}\label{CLT1}
  \frac{1}{\sigma \sqrt{n}} 
  \Bigg[\sum_{j=0}^{n-1} \varphi(X_{tn}) - 
  \int \varphi \;d\mu\Bigg] \xrightarrow{\mathcal D}
  \mathcal N(0,1).
\end{equation*}
\end{conjecture}

Exploring the same ideas from~\cite{MelTor02}, we believe an
Almost Sure Invariance Principle can also be obtained for
the physical measure of this open set of geometric Lorenz
flows with respect to the time-$1$ map.

The paper is organized as follows. In
Section~\ref{sec:statements} we introduce some preliminary
definitions and give the precise statements of our main
results. In Section~\ref{sec:geometr-lorenz-flow} we
construct $C^2$-open sets of Lorenz attractors with smooth
Lorenz one dimensional transformation. Finally, we prove
that these attractors are conjugated to suspension
semiflows with a good hyperbolic structure and, hence, have
exponential decay of correlations.

\section{Setting and statement of results}\label{sec:statements}

Throughout, let $M$ be a compact Riemannian manifold, let
$d$ denote the induced Riemannian distance in $M$,
$\|\cdot\|$ the Riemannian norm and $\Leb$ the induced
normalized Riemannian volume form. 

We will introduce the setting of induced maps and some
concepts from the thermodynamical formalism of suspension
semiflows. Our main results will be stated by the end of the
section. In what follows we write $\|\cdot\|_0$ for the
$\sup$-norm in various functional spaces.

\subsection{Uniformly expanding Markov map}
\label{sec:unif-exp-markov}

We assume that $\cup_{\ell\in L} \De^{(\ell)}$ is an at most
countable partition (Lebesgue modulo zero) of an open domain
$\De \subset M$ by open subsets and let $F:\cup_{\ell\in L}
\De^{(\ell)} \to \De$ be a $C^r$ uniformly expanding Markov
map, $r\ge2$, that is
\begin{enumerate}
\item $F:\De^{\ell}\to\De$ is a $C^{r}$
  diffeomorphism for every $\ell$;
\item there are $C>0$ and $0<\la<1$ such that 
  \begin{enumerate}
  \item for every inverse branch $h_n$ of $F^n$, with
    $n\ge1$,  $d(h_n(x),h_n(y)) \leq C \la^n d(x,y)$; 
  \item if $JF$ is the Jacobian of $F$ with respect to the
    Lebesgue measure, then $\log JF$ is a $C^1$ function and
    $\sup | D((\log JF)\circ h)|\le C$ for every inverse
    branch $h$ of $F$.
  \end{enumerate}
\end{enumerate}
We denote by $\cH_n$ the family of inverse branches of
$F^n$.  In many applications we have that $\Delta$ is a
finite dimensional topological disk.
It is well known that $F$ admits an invariant probability
measure $\nu$ which is absolutely continuous with respect to
Lebesgue.

\subsection{Hyperbolic skew-product structure}
\label{sec:hyperb-skew-product}

We recall some notions previously used by \cite{BaVal2005}
and \cite{AvGoYoc}. We say that the roof function $r: \De
\to \RR^+$ has \emph{exponential tail} if there exists
$\si_0>0$ such that $\int e^{\si_0 r} d\nu <\infty$.

\begin{definition}\label{def:goodroofunc}
  We say that the roof function $r$ is \emph{good} if
\begin{enumerate}
\item $r$ is bounded from below by some positive constant $r_0$;
\item there exists $C>0$ such that $\sup_{h\in\cH_1}\|D(r\circ
  h)\|_{0}\le C <\infty$;
\item it is not possible to write $r=v+u\circ F -u$ on
  $\De$, where $v: \De \to \R$ is constant on each
  $\De^{\ell}$ and $u:\Delta\to\RR$ is a $C^1$-function.
\end{enumerate}
\end{definition}

\begin{remark}\label{rmk:last-condit-definit}
  The last condition in Definition~\ref{def:goodroofunc}
  above corresponds to the uniform non-integrability, or
  aperiodicity, condition defined by Baladi-Vall\'ee in
  \cite{BaVal2005}. There are a number of equivalent
  conditions to this, as proved in
  \cite[Proposition~7.5]{AvGoYoc}.
\end{remark}

Now we define the hyperbolic skew-product structure
with which Lorenz-like flows are endowed.

\begin{definition}
  \label{def:hyp-skew-product} 
  Let $F:\bigcup_{l} \Delta^{(l)} \to \Delta$ be a uniformly
  expanding Markov map, preserving an absolutely continuous
  probability measure $\nu$. A \emph{hyperbolic
    skew-product} over $F$ is a map $\widehat{F}$ from a
  dense open subset of an open domain $\widehat{\Delta}$
  {of a compact Riemannian manifold $M$}, to
  $\widehat{\Delta}$, satisfying the following properties:
\begin{enumerate}
\item there exists a continuous
map $\pi : \widehat{\Delta} \to \Delta$
such that $F\circ \pi = \pi \circ \widehat{F}$ whenever both members
of the equality are defined;
\item there is a $\widehat F$-invariant probability measure
  $\eta$ on $\widehat{\Delta}$, giving full mass to
  $\widehat\Delta$;
\item there exists a family of probability measures
  $\{\eta_x\}_{x\in \Delta}$ on $\widehat{\Delta}$ which is
  a disintegration of $\eta$ over $\nu$, that is, $x\mapsto \eta_x$ is
  measurable, $\eta_x$ is supported on $\pi^{-1}(x)$ and, for
  each measurable subset $A$ of $\widehat{\Delta}$ we have
  $\eta(A)=\int \eta_x(A)\,d\nu(x)$.
  Moreover, this disintegration is smooth: we can find a
  constant $C>0$ such that, for any open subset $V\subset
  \bigcup \Delta^{(l)}$ and for each $u\in C^1(\pi^{-1}(V))$,
  the function $\bar u : V \to \RR, x\mapsto\bar u(x):=\int
  u(y)\,d\eta_x(y)$ belongs to $C^1(V)$ and satisfies
  \begin{align*}
  \sup_{x\in V} \|D\bar u(x)\|
  \leq C \sup_{y\in \pi^{-1}(V)} \|Du(y)\|.
\end{align*}
\item
there is $\kappa>1$ such that, for all $w_1,w_2 \in
\widehat{\Delta}$ in the same leaf,
i.e. $\pi(w_1)=\pi(w_2)$, we have
$
  d(\widehat{F} w_1, \widehat{F} w_2) \leq \kappa^{-1}
  d(w_1,w_2).
$
\end{enumerate}
\end{definition}

\subsection{Good hyperbolic skew-product semiflow}
\label{sec:good-hyperb-semifl}

Now we introduce suspension semiflows over the class of
dynamical systems presented above. Given a function $r :
\cup_{\ell\in L} \De^{(\ell)} \to[r_0,+\infty)$ for some
$r_0>0$ we define
\begin{align*}
  \widehat\De_r=\{(w,t): w \in \widehat\De, \; 0\leq t \le
  r(\pi(w))\}/\sim,
\end{align*}
where $\sim$ is an equivalence relation that identifies the
pairs $(w,r(\pi(w)))$ and $(F(w),0)$. For any $\ell\in L$ let 
$\widehat\De_r^{(\ell)}$ be defined accordingly using $\De_r^{(\ell)}$.
In this way it is natural to consider the suspension semiflow $(Y_t)_t$ given
by $Y_t(w,s)=(w,s+t)$. In these coordinates it coincides
with the flow which consists in the displacement along the
``vertical'' direction. Moreover, we will say that $r$ is
the \emph{roof function} of the suspension skew-product
semi-flow $(Y_t)_t$ over the map $\widehat F$.

If $Y_t$ is a semiflow over a hyperbolic skew-product with a
good roof function which, moreover, has exponential tail,
then we say that $Y_t$ is a \emph{good hyperbolic
  skew-product semi-flow}.
If $\eta$ is an $\widehat F$-invariant probability measure so that
$\int r d\eta<\infty$, then $(Y_t)_t$ preserves the
probability measure $\ov\eta$ given by
$  \ov\eta=(\eta \otimes \Leb)/ \int r \,d\eta.$

\subsection{Statement of results}
\label{sec:statem-result}

Let $C^1(\widehat\Delta_r)$ be the space of bounded observables
$g:\widehat\De_r \to \R$ that are piecewise $C^1$ continuous
(i.e. $C^1$ in each element $\widehat\De_r^{(\ell)}$)
endowed with the norm $ \|g\|_1:=\sup_{w\in \widehat\De} |g(w)|
+ \sup_{w\in \widehat\De} \|Dg(w)\|. $
The following was proved in \cite[Theorem~2.7]{AvGoYoc}.

\begin{theorem}
  \label{thm:AvGoYoc}
  Let $Y_t$ be a good hyperbolic skew-product semi-flow on a
  space $\widehat{\Delta}_r$, preserving the probability
  measure $\ov\eta$. There exist constants $C>0$ and
  $\delta>0$ such that, for each pair of functions
  $\vfi,\psi\in C^1(\widehat{\Delta}_r)$ and $t\geq 0$,
  \begin{align*}
  \left| \int \vfi\cdot \psi \circ Y_t \,d\ov\eta -\left( \int \vfi
  \,d\ov\eta\right) \left( \int \psi \,d\ov\eta \right) \right| \leq C
  \|\vfi\|_1 \|\psi\|_1e^{-\delta t}.
  \end{align*}
\end{theorem}

The main arguments in this paper prove the following.

\begin{maintheorem}
  \label{mthm:Lorenz-hyp-skew}
  Given any compact $3$-manifold $M$, for each $s\ge2$ we
  can find an open subset $\cU$ of $\fX^s(M)$ such that each
  $X\in\cU$ exhibits a geometric Lorenz flow which is
  {$C^s$}-smoothly semi-conjugated to a good hyperbolic skew-product
  semi-flow.
\end{maintheorem}

The meaning of the smooth semi-conjugacy above is: if $U$ is
an open neighborhood such that $\overline{X^t(U)}\subset U$
for all $t>0$ and the attractor is given by
$\Lambda_X:=\cap_{t\ge 0}X^t(U)$, there exists a semi-flow
$Y_t$ on $\widehat\Delta_r$ as stated above, together with a
$C^s$ local diffeomorphism $\phi=\phi_X:\widehat\Delta_r \to M$, whose
image contains an open neighborhood of the geometric Lorenz
attractor $\Lambda_X$, for $X\in\cU$, satisfying $\phi_X\circ
Y_t = X_t \circ \phi_X$ at all points where both sides of the
equality are defined. Moreover $\phi_*\bar\eta=\mu$, where
$\mu$ is the physical measure supported on $\Lambda_X$, see
e.g.~\cite{APPV,AraPac2010}.

\begin{maincorollary}\label{mcor:Lorenz-exp-decay}
  The geometric Lorenz attractors given in
  Theorem~\ref{mthm:Lorenz-hyp-skew} have exponential decay
  of correlations for $C^1$ observables.
\end{maincorollary}

Indeed, if we take $\vfi,\psi$ a pair of $C^1$ functions on
$\Lambda_X$, for some $X\in\cU$, then these maps are
restriction of $C^1$ maps on an open neighborhood $W$ of
$\Lambda_X$ in $M$, which we denote by the same letters. We
can assume that $W$ contains the image
$\phi(\widehat\Delta_r)$ of $\phi$, for otherwise we can
extend $\vfi,\psi$ to this neighborhood using bump functions
without changing their values over $\Lambda_X$. Hence
$\bar\vfi:=\vfi\circ\phi$, $\bar\psi:=\psi \circ\phi$ are
$C^1$ functions on $\widehat\Delta_r$ and
\begin{align*}
  \int \bar\vfi\cdot \bar\psi \circ Y_t \,d\ov\eta
  &=
  \int (\vfi\circ\phi)\cdot (\psi\circ\phi \circ Y_t)
  \,d\ov\eta
  =
  \int (\vfi\circ\phi)\cdot (\psi\circ X_t \circ \phi)
  d\,\ov\eta
  \\
  &=
  \int \vfi\cdot (\psi \circ X_t) \,d(\phi_*(\ov\eta))
  =
  \int \vfi\cdot (\psi \circ X_t) \,d\mu,
\end{align*}
so the exponential decay of correlations follows from
Theorem~\ref{thm:AvGoYoc}.

\subsection{Strategy of the proof}
\label{sec:strategy-proof-theor}

The proof of the exponential decay of correlations for
singular flows stated in
Corollary~\ref{mcor:Lorenz-exp-decay} for the class of
Lorenz attractors associated to vector fields in the open
sets of Theorem~\ref{mthm:Lorenz-hyp-skew} consists of three
main steps.

First we prove that every three-dimensional manifold admit
Lorenz-like attractors whose one-dimensional stable
foliation is $C^2$. Moreover, such a contruction is robust
in the sense that it holds for every $C^2$ close vector
field.  In fact we get that the restriction of the flows to
the Lorenz-like attractors are smoothly semi-conjugated to
suspension over a $C^2$ non-uniformly hyperbolic hyperbolic
Poincar\'e map, that is, a partially hyperbolic
transformation with one-dimensional stable direction and
one-dimensional central direction with positive Lyapunov
exponent with respect to an absolutely continuous invariant
measure.  The crucial regularity of the stable foliation is
a consequence of a strong partially hyperbolicity of the
Poincar\'e transformation requires a condition on the
eigenvalues of the singularity, which is a $C^1$-open
condition in a neighborhood of the original vector field.

The second ingredient is to prove that such flows are
semi-conjugated to suspension flows whose roof function is
constant along stable leaves. Such a property is obtained
first for the geometric Lorenz attractor by
construction. Although all sufficiently close vector fields
still preserve the global cross-section for the original
flow, it is most likely that the Poincar\'e return time to the
original cross-section of those flows not to be constant along the
stable leaves. So, for every sufficiently close flow we
consider adapted sections, that is, sections that are
foliated by the one-dimensional strong-stable leaves (see
e.g.  \cite{APPV}). Hence, the invariance of the strong
stable foliation guarantees that the first return time
function is constant on stable leaves.

Finally we show that these flows satisfy the assumptions of
\cite{BaVal2005,AvGoYoc} on a criterium for exponential
decay of correlations. Namely, the flow is conjugated to a
suspension flow over a hyperbolic skew-product that admits a
unique SRB measure with a smooth disintegration along the
strong-stable foliation, and whose induced roof function
satisfies a non-integrability condition and has exponential
tail. For the later we revisit the construction of the SRB
measure for Lorenz attractors and obtain a disintegration of
the measure as fixed points associated to suitable transfer
operators.

%

\subsection*{Acknowledgements}
Part of this work was done during the stay of P.V. at the
conference Low-Dimensional Dynamics as an activity of the
CODY Autumn in Warsaw, whose excellent research conditions
are greatly acknowledged. {The authors are deeply grateful
to the anonimous referee for many suggestions and comments that helped
to improve the manuscript.}


\section[Smooth Lorenz map]{The geometric
  Lorenz flow with smooth Lorenz map}
\label{sec:geometr-lorenz-flow}

Here we describe the construction of geometric Lorenz flows
with $C^k$ smooth strong-stable foliation, for each integer
$k\ge2$, following \cite[Chap. 3, Sect. 3.3]{AraPac2010} and
taking advantage of the idea of $k$-domination
from~\cite{HPS77}. The proof of the next proposition will be
given along the rest of this section.

\begin{proposition}
  \label{pr:Lorenz-smooth-stable}
  Given $k\in\ZZ^+$ there exists a $C^k$ vector
  field $X$ on $\RR^3$ and a $C^k$ neighborhood $\cU$ of $X$
  in $\fX^k(\RR^3)$ such that
  \begin{itemize}
  \item there exists a trapping region $U$ containing a
    surface cross-section $S$ for the flow of every
    $Y\in\cU$;
  \item the maximal positively invariant subset $\Lambda_Y$
    inside $U$ for the flow of $Y$ is a transitive attractor
    containing a hyperbolic singularity $\sigma$;
  \item the first return map $P_Y$ from $S^*\subset S$
    to $S$ admits a $C^k$ smooth uniformly contracting
    foliation $\cF_Y$, where $S^*=S\setminus W^{s}_{loc}(\sigma)$;
  \item the induced one-dimensional quotient map
    $f_Y=P_Y/\cF_Y$ is a piecewise $C^k$ smooth expanding
    map with two branches defined on intervals $I^\pm$,
    where $|Df_Y|>\sqrt2$ having a common boundary point
    $0$, in a neighborhood of which the derivative $Df_Y$
    grows as the logarithm of the distance to $0$;
  \item the map $f_Y$ is locally eventually onto: every
    interval $J$ in the domain of $f_Y$ admits a subinterval $J_0$
    and some iterate $n>2$ such that $f^n(J_0)$ contains
    either $I^-$ or $I^+$.
  \end{itemize}
\end{proposition}

A similar idea in the flow setting was
used in~\cite{Ro93} to obtain a $C^2$ strong-stable
foliation on a singular attractor.

\subsection{Near the singularity}\label{sec:near-singul}

In a neighborhood of the origin we consider the linear
system $(\dot x, \dot y, \dot z)=(\lambda_1 x,\lambda_2 y,
\lambda_3 z)$, thus
\begin{align}\label{eq:LinearLorenz}
X^t(x_0,y_0,z_0)=
(e^{\lambda_1t}x_0, e^{\lambda_2t} y_0, e^{\lambda_3t}z_0),
\end{align}
where $\lambda_2<\lambda_3<0<-\lambda_3<\lambda_1$ and
$(x_0, y_0, z_0)\in\RR^3$ is an arbitrary initial point near
the origin.

To ensure that arbitrarily small $C^2$ perturbation of this
flow are still smoothly linearizable near the continuation
of the hyperbolic singularity $\sigma$, we use a smooth
linearization result which can be found in Hartman
\cite[Theorem 12.1, p. 257]{Hartman02}.

\begin{theorem}
  \label{thm:smooth-linear}
  Let $n\in\ZZ^+$ be given. Then there exists an integer
  $N=N(n)\ge2$ such that: if $\Gamma$ is a real non-singular
  $d\times d$ matrix with eigenvalues
  $\gamma_1,\dots,\gamma_d$ satisfying
  \begin{align}\label{eq:non-resonance}
    \sum_{i=1}^d m_i \gamma_i \neq \gamma_k
    \quad
    \text{for all}
    \quad
    k=1,\dots, d
    \qand
    2\le\sum_{j=1}^d m_j\le N
  \end{align}
  and if $\dot\xi=\Gamma\xi+\Xi(\xi)$ and
  $\dot\zeta=\Gamma\zeta$, where $\xi,\zeta\in\RR^d$ and
  $\Xi$ is of class $C^N$ for small $\|\xi\|$ with
  $\Xi(0)=0, \partial_\xi\Xi(0)=0$; then there exists a
  $C^n$ diffeomorphism $R$ from a neighborhood of $\xi=0$ to
  a neighborhood of $\zeta=0$ such that $R\xi_t
  R^{-1}=\zeta_t$ for all $t\in\RR$ and initial conditions
  for which the flows $\zeta_t$ and $\xi_t$ are defined in
  the corresponding neighborhood of the origin.
\end{theorem}
Hence it is enough for us to choose the eigenvalues
$(\lambda_1,\lambda_2, \lambda_3)\in\RR^3$ satisfying a
\emph{finite set of non-resonance relations}
\eqref{eq:non-resonance} for a certain $N=N(2)$. For this
condition defines an open set in $\RR^3$ and so all small
$C^1$ perturbations $Y$ of the vector field $X$ will have a
singularity whose eigenvalues
$(\lambda_1(Y),\lambda_2(Y),\lambda_3(Y))$ are still in the
$C^2$ linearizing region.

We consider the set $S=\{ (x,y,1) : |x|\le1/2, |y|\le1/2\}$ and
\begin{align*}
S^-&=\big\{ (x,y,1)\in S : x<0 \big\},&
\qquad
S^+&=\big\{ (x,y,1)\in S : x>0 \big\}\quad\text{and}
\\
S^*&=S^-\cup S^+=S\setminus\Gamma,\quad \text{where}&
\Gamma&=\big\{(x,y,1)\in S : x=0 \big\}.
\end{align*}
We assume without loss of generality that $S$ is a transverse
section to the flow so that every trajectory eventually
crosses $S$ in the direction of the negative $z$ axis as in
Figure~\ref{L3Dcusp}. Note that $\Gamma$ is the intersection
of $S$ with the local stable manifold $W^s_{loc}(\sigma)$ of
the equilibrium $\sigma=(0,0,0)$: $S^*=S\setminus
W^s_{loc}(\sigma)$.
Hence we get
\begin{align*}
X^\tau(x_0,y_0,1)=
\big( \sgn(x_0),  y_0e^{\lambda_2\tau(x_0)},
e^{\lambda_3\tau(x_0)}\big)
=
\big( \sgn(x_0),
y_0|x_0|^{-\frac{\lambda_2}{\lambda_1}},
|x_0|^{-\frac{\lambda_3}{\lambda_1}}\big)
\end{align*}
where $\sgn(x)=x/|x|$ for $x\neq0$, and
$0<\alpha=-\frac{\lambda_3}{\lambda_1} <1
<\beta=-\frac{\lambda_2}{\lambda_1}$ by the choice of the
eigenvalues.

Consider also $\Sigma=\{ (x,y,z) : |x|=1
\}={\Sigma}^-\cup{\Sigma}^+$ with ${\Sigma}^{\pm}=\{ (x,y,z)
: x=\pm 1\}$. For each $(x_0,y_0,1)\in S^*$ the time
$\widetilde\tau$ such that
$X^{\widetilde\tau}(x_0,y_0,1)\in\Sigma$ is given by
$\widetilde\tau(x_0)=-\frac{1}{\lambda_1}\log{|x_0|}$, which
depends on $x_0\in S^*$ only and is integrable with respect
to Lebesgue measure on any bounded interval $J$ of the real
line: $0<\int_J\widetilde\tau(x_0)\,d\lambda(x_0)<\infty$;
see \cite{AraPac2010,PT09}.  Moreover we also have that
$\widetilde\tau$ is bounded from below by
$\tau_0:=\log2/\lambda_1>0$.
Let $L:S^*\to\Sigma$ be given by
\begin{align}\label{eq:passcross}
  L(x,y,1)=\big( \sgn(x),y|x|^\beta,|x|^\alpha \big).
\end{align}
Clearly each line segment $S^*\cap\{x=x_0\}$ is
taken to another line segment $\Sigma\cap\{z=z_0\}$ as
sketched in Figure~\ref{L3Dcusp}.

\begin{figure}[h]
\includegraphics[width=6cm]{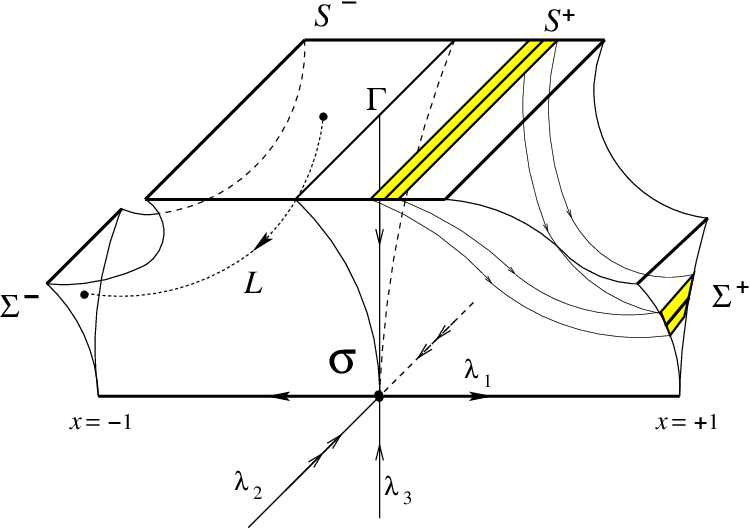}
\caption{\label{L3Dcusp}Behavior near the origin.}
\end{figure}

\subsection{The rotating effect}
\label{sec:rotating-effect}

To imitate the random turns of a regular orbit around the
origin and obtain a butterfly shape for our flow, as in the
original Lorenz flow,
the sets $\Sigma^\pm$ should return to the cross-section
$S$ through a flow described by a suitable composition
of a rotation $R_\pm$, an expansion $E_{\pm\theta}$ and a
translation $T_\pm$.

We assume that the ``triangles'' $L(S^\pm)$ are compressed
in the $y$-direction and stretched on the other transverse
direction and that this return map takes line segments
$\Sigma\cap\{z=z_0\}$ into line segments $S\cap\{x=x_1\}$,
as sketched in Figure~\ref{L3D}.

\begin{figure}[h]
\includegraphics[width=5cm]{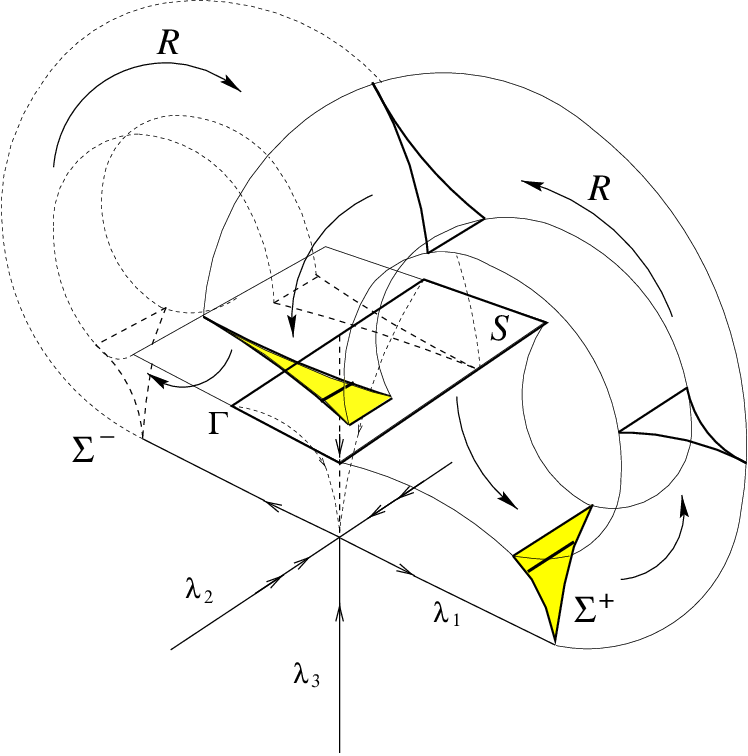}
\caption{\label{L3D}$R$ takes $\Sigma^\pm$ to $S$.}
\end{figure}
The choice of
$R_\pm, T_\pm, E_{\pm\theta}$ can be seen in \cite[Chapter
3, Section 3]{AraPac2010} of \cite{galapacif09}.

These transformations $R_\pm , E_{\pm\theta},
T_\pm$ take line segments $\Sigma^\pm\cap\{z=z_0\}$
into line segments $S\cap\{x=x_1\}$ as shown in
Figure~\ref{L3D}, and so does the composition $T_\pm\circ
E_{\pm\theta}\circ R_\pm$.

This composition of linear maps describes a vector field in
a region outside $[-1,1]^3$, in the sense that one can use
the above linear maps to define a vector field $X$ such that
the first return map to $S$ of the associated flow realizes
$T_\pm\circ E_{\pm\theta}\circ R_\pm$ as a map
$\Sigma^\pm\to S$.

We note that the flow on the attractor we are constructing
will pass through the region between $\Sigma^\pm$ and $S$ in
a relatively small time with respect the the linearized
region. The linearized regions will then dominate all
estimates of expansion/contraction.

More precisely, the time a point $(x_0,y_0,1)\in S^*$ takes
to return to $S$ is given by
$\tau_X(x_0,y_0)=\tau(x_0)=\widetilde\tau(x_0)+s_0=-(1/\lambda_1)\log|x_0|
+ s_0$, where $s_0>0$ is the constant time ``lateral
triangles'' $\Sigma^\pm$ take to flow until $S$. Hence the
return time to $S$ is clearly dominated by the behavior of
$\widetilde\tau$ and the behavior of the derivative of the
flow is dominated by the behavior of the flow in the
linearized region.

\subsection{The first return map to $S$}
\label{sec:first-return-map}

The above combined effects imply that 
the foliation of $S$ given by the lines $S\cap\{x=x_0\}$ is
invariant under the Poincar\'e first return map $P:S^*\to
S$, meaning that, for any given leaf $\gamma$ of this
foliation, its image $P(\gamma)$ is contained in a leaf of
the same foliation. Hence $P$ must have the form
$P(x,y)=\big( f(x),g(x,y)\big)$ for some functions
$f:I\setminus\{0\}\to I$ and $g:(I\setminus\{0\})\times I\to
I$, where $I=[-1/2,1/2]$.

Taking into account the definition of $L$ from the
linearized region we see that
\begin{align*}
  f(x)=\left\{
\begin{array}{cccc}
  f_1(x^\alpha), & \text{if   }  x < 0       \\
  f_0(x^\alpha), & \text{if   } x > 0
\end{array};
\right.
\quad \mbox{with $f_i =(-1)^{i}a\cdot x+b_i$}\quad  i=0,1;
\end{align*}
for constants $a>0$ and $b_0,b_1\in(-1/2,1/2)$, and
\begin{align*}
  g(x,y)=\left\{
\begin{array}{cccc}
g_1(x^\alpha,y\cdot x^\beta), & \text{if   } x < 0       \\
g_0(x^\alpha,y\cdot x^\beta), & \text{if   } x > 0
\end{array}
\right.,
\end{align*}
where $g_1|I^-\times I\to I $ and $g_0|I^+\times I\to I$ are
suitable affine maps, with $I^-=[-1/2,0)$, $I^+=(0,1/2]$.

\subsection{Properties of the one-dimensional map $f$}
\label{sec:propert-one-dimens}

Now we specify the properties of the one-dimensional map $f$
that follow from the previous construction. On the one hand
\begin{enumerate}
\item[(f1)] to imitate the symmetry of the Lorenz equations
  we take $f(-x)=-f(x)$. This is not essential in what
  follows and is not preserved under perturbation of the
  flow;
\item[(f2)] $f$ is discontinuous at $x=0$ with lateral limits
  $f(0^-)=+\frac12$ and $f(0^+)=-\frac12$ 
\end{enumerate}
Hence
\begin{align*}
  f(0+)=b_1=-\frac12, \quad f(0-)=b_0=\frac12
  \qand f\left(\frac12\right)=\frac{a}{2^\alpha}+b_1\le\frac12,
\end{align*}
thus $0<a\le 2^{\alpha}$. Since $Df(x)=a
\alpha|x|^{\alpha-1}$, its minimum is $DF(1/2)=a \alpha
2^{1-\alpha}$ and to get $Df>1$ we must have $a\alpha
2^{1-\alpha}>1$.
\begin{enumerate}
\item[(f3)] The map $f$ is differentiable on
  $I\setminus\{0\}$ and $Df(x)>\sqrt{2}$;
\end{enumerate}
to get this all we need is to choose
\begin{align}\label{eq:alphachoice1}
  \frac{2^{\alpha-1/2}}\alpha < a < 2^\alpha
  \quad\text{so that}\quad
  2^{1/2}\alpha>1
  \quad\text{or}\quad
  \alpha>1/\sqrt2.
\end{align}
\emph{This imposes a restriction on $\alpha=-\lambda_3/\lambda_1$,
thus the eigenvalue $\lambda_3$ cannot be too small with
respect to the eigenvalue $\lambda_1$ at the singularity.}
\begin{enumerate}
\item[(f4)] the lateral limits of $Df$ at $x=0$ are
  $Df(0^-)=+\infty$ and $Df(0^+)=-\infty$.
\end{enumerate}

On the other hand $g:S^*\to I$ is defined in such a way that
it contracts the second coordinate: $g'_y(w)\le\mu<1$ for
all $w\in S^*$, and the rate of contraction of $g$ on the
second coordinate should be much higher than the expansion
rate of $f$.  In addition the expansion rate is big enough
to obtain a strong mixing property for $f$.

\begin{remark}
  \label{rmk:non-degenerate}
  The expression of $Df$ ensures that the map $f$ satisfies
  \begin{align*}
    |\log Df(x)- \log Df(y)|
    = 
    (1-\alpha)\log\left|\frac{y}{x}\right|
    =
    (1-\alpha)\log\left|\frac{y-x}{y}+1\right|
    \le
    \frac{1-\alpha}{|x|}|y-x| ;
  \end{align*}
  and also that $Df(x)=a\alpha|x|^{\alpha-1}$
  which shows that $f$ behaves like a power of the distance
  to the singular set $\{0\}$.
\end{remark}



\subsection{Properties of the map $g$}
\label{sec:propert-map-g}

We note that by its definition the map $g$ is piecewise
$C^2$ and we can obtain the following bounds on its partial
derivatives:
\begin{enumerate}
 \item For all $(x,y)\in S^*$ with $x \neq 0$, we have
$|{\partial_y} g(x,y)|=  |x|^\beta$. As $\beta>1$ and $|x|\leq 1/2$
there is $0<\lambda<1$ such that
\begin{equation}
 \label{gy}
|{\partial_y} g| < \lambda.
\end{equation}
\item For $(x,y)\in S^*$ with $x \neq 0$, we have
${\partial_x} g(x,y)=\beta x^{\beta-\alpha}$.
Since $\beta>\alpha$  and $|x|\leq 1/2$  we get
$|\partial_x g| < \infty.$
\end{enumerate}
We note that from the first item above we have, first, a
very strong domination of the contraction along the
$y$-direction over the expansion along the $x$-direction,
that is
\begin{align}\label{eq:domination}
  \frac{|\partial_y g(x,y)|}{|Df(x)|}
  \approx
  |x|^{\beta-\alpha+1}
  \approx
  \frac{|x|^\beta}{|Df(x)|}
  \quad\text{with}\quad
  \beta-\alpha+1>1.
\end{align}
Secondly, from this there follows the uniform contraction of
the foliation $\cF_X$ given by the lines
$S\cap\{x=constant\}$, that is: there exists a constant
$C>0$ such that, for any given leaf $\gamma$ of the
foliation and for $y_1,y_2\in\gamma$, then
\begin{align*}
  \dist\big(P^n(y_1),P^n(y_2)\big)\le C\lambda^n\dist(y_1,y_2) 
  \quad\text{when}\quad n\to\infty.
\end{align*}

Thus the study of the maximal invariant set $\Lambda$ inside
the trapping region
\begin{align}\label{eq:trap-reg-U}
  U:=\{X^t(x,y,1):(x,y,1)\in S, 0\le t\le \tau_X(x,y)\}\cup\{(0,0,0)\}
\end{align}
for this $3$-flow can be reduced to the study of a
bi-dimensional map, where $\tau_X$ is the first return time
of the orbit of $(x,y,1)\in S$ under $X^t$ to $S$.
Moreover, the dynamics of this map can be further reduced to
a one-dimensional map, because the invariant contracting
foliation $\cF_X$ enables us to identify two points on the same
leaf, since their orbits remain forever on the same leaf and
the distance of their images tends to zero under
iteration. See Figure \ref{L1D} for a sketch of this
identification.

The quotient map $f:S^*/\cF_X\to S/\cF_X$ obtained through
the identification $\pi:S\to S/\cF_X$ will be called
\emph{the (one-dimensional) Lorenz map}. It satisfies
$f\circ\pi=\pi\circ P$ by construction and we note that
$S^*/\cF_X$ is naturally identified with $I\setminus\{0\}$
through a diffeomorphism, so that we obtain in fact the
first component of the Poincar\'e return map $P$ to
$S$. Figure~\ref{L1D} shows the graph of this
one-dimensional transformation.

\begin{figure}
\psfrag{P}{$P_0$}\psfrag{Q}{$P_1$}
\includegraphics[width=4.5cm]{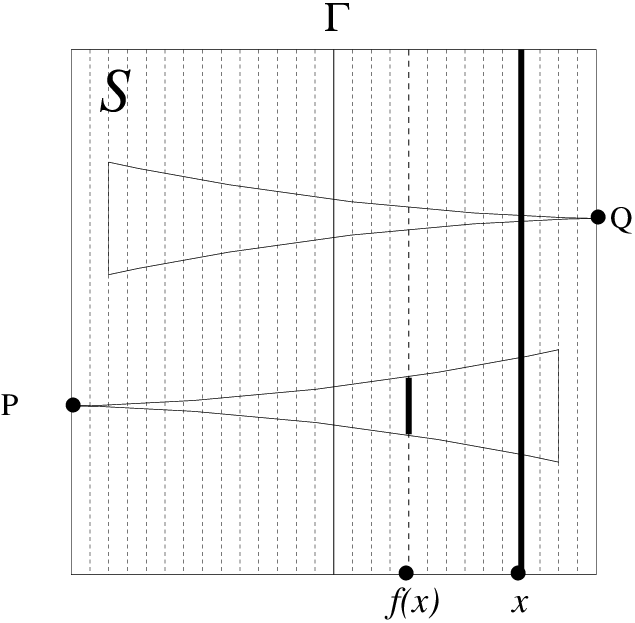}
$\qquad$
\includegraphics[width=4.4cm]{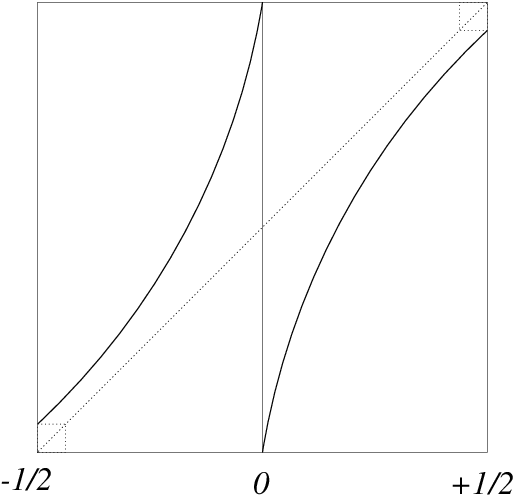}
\caption{\label{L1D} On the left: projection on $I$ through
  the stable leaves and a sketch of the image of one leaf
  under the return map. On the right: the Lorenz map $f$.}
\end{figure}

\subsection{Persistence and smoothness of the contracting
  foliation}
\label{sec:persist-contract-fol}

The following persistence property is a consequence of the
domination of the contraction along the $y$-direction over
the expansion along the $x$-direction (see
e.g. \cite{AraPac2010}).

\begin{theorem}
  \label{thm:persistent-foliation}
  Let $X$ be the vector field obtained in the construction
  of the geometric Lorenz model and $\cF_X$ the invariant
  contracting foliation of the cross-section $S$. Then any
  vector field $Y$ which is sufficiently $C^1$-close to $X$
  admits an invariant contracting continuous foliation
  $\cF_Y$ on the cross-section $S$ with $C^1$ leaves.
\end{theorem}

It can be shown that the holonomies along the leaves are in
fact H\"older-$C^1$ (see e.g \cite{AraPac2010}).  Moreover,
under a \emph{strong dissipative} condition on the
eigenvalues of the equilibrium $\sigma$, that is,
$\beta>\alpha+k$ for some $k\in\ZZ^+$ (recall equation
\eqref{eq:passcross}) it can be proved, following
\cite{HPS77} and \cite{Ro93}, that $\cF_Y$ is then a $C^k$
smooth foliation and so the holonomies along its leaves are
$C^k$ maps. This is just an application of $k$-normal
hyperbolicity to flows.  We deduce the following.

\begin{theorem}\label{thm:smooth-ss}
  For strongly dissipative Lorenz attractors, with
  $\beta>\alpha+k$, the one-dimensional quotient map $f$ is
  $C^k$ smooth away from the singularity.  Moreover, this
  smoothness property is also persistent for all nearby
  $C^k$ flows, since the condition $\beta>\alpha+k$ is open
  in the $C^1$ topology.
\end{theorem}

In what follows we fix $k\ge2$, which guarantees that the
foliation $\cF_Y$ is $C^2$ for all $Y$ in a $C^2$
neighborhood $\cU$ of $X$. Moreover the one-dimensional
piecewise expanding map $f_Y$ given as the quotient map of
the corresponding Poincar\'e map $P_Y$ over the leaves of
the foliation $\cF_Y$ associated to $Y\in\cU$.

\subsection{Robust transitivity properties}
\label{sec:robust-transit-prope}

Here we show that the Lorenz flows constructed are suspension flows with a roof
function that is constant along the stable foliation on the cross-section $S$.
First we discuss the case of the geometric Lorenz flow $X^t$ constructed in
the previous section and then complete the proof of Proposition~\ref{pr:Lorenz-smooth-stable}.
 
First, we observe that $S$ is, for the  geometric
Lorenz flow $X^t$ constructed in the previous sections, a
collection of strong-stable leafs of the flow. Indeed, we
can write
\begin{align*}
  S=\bigcup_{-1/2\le x\le 1/2} W^{ss}_{X,1/2}(x,0,1)
\end{align*}
as a family of local strong-stable leaves for the vector
field $X$ with radius $1/2$ through the points in
$I\times\{0\}\times\{1\}$, the orthogonal segment to
$\Gamma$ in $S$;  see the left hand side of Figure~\ref{L1D}.
Since the strong-stable foliation $\cF_Y$ for $C^2$ close
vector fields $Y$ is also a $C^2$ foliation, then we can
repeat the construction with respect to every close vector field $Y$
and obtain a smooth surface
\begin{align*}
  S_Y:=\bigcup_{-1/2\le x\le 1/2} W^{ss}_{Y,1/2+\epsilon}(x,0,1)
\end{align*}
which is a cross-section for the flow $Y^t$, and contains
the continuation of the points $P_1,P_0$ as the first visits
of the branches of the unstable manifold $W^u_Y(\sigma_Y)$
of the singularity to the cross-section, for small enough
$\epsilon>0$; see Figure~\ref{L1D}.


Moreover we can, by the $C^2$ change of coordinates which
linearizes the flow around $\sigma_Y$, assume
first that the new singularity $\sigma_Y$ is still at the
origin, and that on $S_Y$ we have coordinates $(x,y)$ such
that $x=\const$ represents a curve on $S_Y$ which is
uniformly contracted by the Poincar\'e return map associated
to $Y$; in fact, $x=\const$ is precisely the strong-stable
manifold through $(x,0,1)$.
In particular, \emph{this ensures that the Poincar\'e return
  time of points on $S_Y$ to $S_Y$ in the same leaf of
  $\cF_Y$ is constant}, for all $Y$ sufficiently $C^2$ close
to $X$. This extends a very useful property of $X$ to all
nearby vector fields.

In what follows $|x|$ still represents the distance of the
curve $x=\const$ to $\Gamma_Y:=S_Y\cap W^{s}_{loc}(\sigma_Y)$,
the intersection of the local stable manifold of the
singularity with $S_Y$. The Poincar\'e return time of a
point $(x_0,y_0)$ in $S_Y^*:=S_Y\setminus\Gamma_Y$ equals
\begin{align}\label{eq:tau-pert}
  \tau_Y(x_0,y_0)=-\frac{1}{\lambda_1(Y)}\log|x_0| + s_Y(x_0),
\end{align}
where $s_Y:S_Y^*\to\RR$ is $C^2$ close to $s_X=\const$.

Let $f_Y$ be the one-dimensional map as the quotient
map of the corresponding Poincar\'e map $P_Y$ over the leaves
of the foliation $\cF_Y$ for all flows $Y$ close to $X$ in
the $C^1$ topology as above. 
Since the leaves of $\cF_Y$ are $C^1$ close to
those of $\cF$, it follows that $f_Y$ is $C^1$ close to $f$
and thus there exists $c\in[-1/2,1/2]$ which plays for $f_Y$
the same role of the singular point at $0$, so that, after a
linear change of coordinates, we can assume that $c=0$ and
properties (f2)-(f4) from
Subsection~\ref{sec:propert-one-dimens} are still valid,
albeit with different constants, for $f_Y$ on a subinterval
$[-b_0,b_1]$ for some $0<b_0,b_1<1/2$ close to $1/2$. 
In particular
\begin{align}\label{eq:Dfapprox}
  Df_Y(x)\approx |x|^{\alpha-1} \quad
  \text{i.e.}\quad
  \frac1C\le\frac{Df_Y(x)}{|x|^{\alpha-1}}\le C
\end{align}
for some $C>1$ uniformly on a $C^2$ neighborhood of $X$,
where $\alpha=\alpha(Y)=-\lambda_3(Y)/\lambda_1(Y)$
depends smoothly on vector field.
Finally, the condition (f3) ensures that $f_Y$ has enough expansion
to easily prove that \emph{every $f_Y$ is locally eventually
  onto for all $Y$ close to $X$}. More precisely,
  
\begin{lemma}{\cite[Lemma 3.16]{AraPac2010}}
  \label{le:leo}
  For any interval $J\subset(-b_0,b_1)$ there exists an iterate
  $n\ge1$ such that $f_Y^n(J)=(0,b_1)$ or $f_Y^n(J)=(-b_0,0)$;
  and the next iterate covers $(f(-b_0),f(b_1))$.
\end{lemma}

In particular, this implies that $f_Y$ is transitive and, as
well known, $\Lambda_Y$ turns out to be
transitive also. \emph{So we have a robust transitive
  attractor on a $C^2$ neighborhood of $X$.}

\begin{remark}
  \label{rmk:densepreorbit}
  Lemma~\ref{le:leo} implies also that every given point $q$
  of $(-b_0,b_1)$ belongs to the some positive image $f^n(J)$ of
  any given interval $J\not\ni0$, for some $n>0$. Hence
  \emph{every $q\in(-b_0,b_1)$ has a dense set of pre-images},
  that is $\overline{\cup_{k\ge0} f_Y^{-k}(\{q\})}=[-b_0,b_1]$.
\end{remark}

Taken together with the results in the previous subsections we proved
Proposition~\ref{pr:Lorenz-smooth-stable}.

\begin{remark}
  \label{rmk:delta-dense}
  Given $\delta>0$ there exists an integer $N$ such that
  $O_N(Y):=\cup_{i=1}^N (f_Y^i)^{-1}(\{0\})$ is
  $2\delta$-dense in $(-b_0,b_1)$. Moreover, by slightly
  modifying the return to $S$ from $\Sigma^{\pm}$ (see
  Section~\ref{sec:rotating-effect}) during the construction
  of $X$, we can assume that the singular point $0$ of $f_Y$
  does not belong to $O_N(Y)$ and $O_N(Y)$ is $\delta$-dense
  in $I$, for each vector field $Y$ in a $C^2$ neighborhood
  $\cU$ of the original geometric Lorenz flow $X$.
\end{remark}

We can, by another smooth change of coordinates, assume
without loss of generality that both $b_0,b_1$ equal $1/2$
in what follows.

\section[Good hyperbolic skew-product]{The geometric Lorenz
  flow is a good hyperbolic skew-product}
\label{sec:geometr-lorenz-flow-1}

Now we explain step by step how to obtain all the required
properties to conclude that the geometric Lorenz flows
constructed in Section~\ref{sec:geometr-lorenz-flow} are
good hyperbolic skew-products.

We fix $Y$ a vector field in a $C^2$ neighborhood $\cU$ of
$X$, with a geometric Lorenz attractor $\Lambda$ in a
trapping region $U$, which contains the image of the
cross-section $S$ under the flow $Y^t$ until its points
first return to $S$. We denote by $P_Y:S_Y^*\to S_Y$ the
associated Poincar\'e first return map, and by
$f_Y:I\setminus\{0\}\to I$, with $I=[-1/2,1/2]$, the
corresponding Lorenz map obtained by the action of $P_Y$ on
the leaves $\cF_Y$ of the contracting foliation on $S^*=S^*_Y$. We also set
$\tau=\tau_Y:S^*\to[\tau_0,+\infty)$ the first return time
function associated to $P=P_Y$ so that $P(w)=Y^{\tau(w)}(w),
w\in S^*$.

\subsection{The uniformly expanding Markov map}
\label{sec:uniformly-expand-mar}

We show that a Lorenz map $f=f_Y$ admits an induced map $F$
on a small interval $\Delta\subset I$ which is a uniformly
expanding Markov map. Induced Markov transformations for
$C^{1+\alpha}$ Lorenz transformations were obtained first by
\cite{KDO06} but that construction is not suitable for our
estimates. Indeed, a crucial property, the
  uniform non-integrability, is obtained in
  Subsection~\ref{sec:uniform-non-integr} using essentially
  that $Dr(x)\approx|x|^{-1}$ while
  $Df(x)\approx|x|^{\alpha-1}$ for $x$ near the zero, so
  their growth rates are significantly different near the
  singularity. To take advantage of this we build an induced
map on an interval $\Delta$ which is an open neighborhood of
the singular point at the origin. To the best of our
knowledge this type of construction is not available in the
literature, but can be obtained using a number of other
results as follows.

We assume that $f$ satisfies the
properties given in Section~\ref{sec:propert-one-dimens} and
follow the exposition in
\cite{alves-luzzatto-pinheiro2004,gouezel}.  Let
$b$ be a fixed constant satisfying $0 < b <
\min\{1/2,1/(4|1-\alpha|)\}$.

We start by observing that $f$ is a local diffeomorphism
away from $0$ and behaves like a power of the distance to
the singular set; see Remark~\ref{rmk:non-degenerate}. This
says that $0$ is a ``non-degenerate singularity'' according
to \cite{ABV00}, a concept generalized from one-dimensional
dynamics.

Given $\xi,\sigma\in(0,1)$ and $\delta>0$, we say that $n$
is a {\em $(\sigma,\delta)$-hyperbolic time} for a
point $x\in I$ if, for all $1\le k \le n$
\begin{align*}
  \prod_{j=n-k}^{n-1}\|Df(f^j(x))^{-1}\| \le \sigma^{ k} \qand
  |f^{n-k}(x)|_\delta\ge \sigma^{b  k},
\end{align*}
where $|z|_\delta=|z|$ if $|z|<\delta$ and $|z|_\delta=1$ otherwise.

We present well-established results showing that (i) if \( n
\) is a hyperbolic time for \( x \), the map \( f^{n} \) is
a diffeomorphism with uniformly bounded distortion on a
neighborhood of \( x \) that is mapped to a disk of uniform
radius; (ii) Lebesgue almost every point has many hyperbolic times
for the one-dimensional Lorenz transformation $f$.  We say
that $f^n$ has \emph{bounded distortion} by a factor $D$ on
a set $V$ if, for every $x,y\in V$,
$$ \frac{1}{  D} \le \frac{|Df^n(x)|}{|Df^n (y)|}\le D \,.  $$

\begin{lemma}\label{le:hyp-time1}
  Given $\sigma\in(0,1)$ and $\delta>0$, there
exist $\delta_1, D_{1},\kappa   >0$ (depending only on
$\sigma,\delta$ and on the map $f$) such that for any $x\in M$ and
\( n\geq 1 \) a $(\sigma,\delta)$-hyperbolic time for \( x \),
there exists a neighborhood \( V_n(x) \) of \( x \) with the
following properties:
\begin{enumerate}
\item $f^{n}$ maps $V_n(x)$ diffeomorphically onto the ball
$B(f^{n}(x), \delta_1 )$;
\item for $1\le k <n$ and $y, z\in V_n(x)$, $
  |f^{n-k}(y)-f^{n-k}(z)| \le \sigma^{ k}|f^{n}(y)-f^{n}(z)|$;
\item $f^n$ has distortion bounded by $ D_1$ on
  $V_n(x)$;
\item 
\(V_n(x) \subset B(x, 2\delta_1 \sigma^{ n}) \).
\end{enumerate}
\end{lemma}

\begin{proof}
  For the proofs of items 1, 2, 3 see Lemma 5.2 and
  Corollary 5.3 in \cite{ABV00}.  Item~4 is an immediate
  consequence of the backward contraction property at item~2.
\end{proof}

We say that the sets \( V_n(x) \) are \emph{hyperbolic
  pre-balls} and their images \( f^{n}(V_n(x)) \) are
\emph{hyperbolic balls}; the latter are indeed balls of
radius \( \delta_1 \).
For the existence of hyperbolic times, we observe that
\begin{itemize}
\item $f$ is a $C^2$ piecewise expanding map since $|Df|>\sqrt2>1$;
\item It follows from \cite{KDO06} that $C^{2}$ Lorenz
  transformations have an unique absolutely continuous
  invariant measure $\nu_0$ with Lyapunov exponent
  $\lambda(\nu_0)=\int \log|f'| \, d\nu_0>0$ whose basin
  $B(\nu_0)$ covers Lebesgue almost every point;
\item Since $|Df(x)|\approx |x|^{\alpha-1}$ behaves like a power of the distance to the singular point 
then $\log|x|$ is $\nu_0$-integrable and, by the Ergodic Theorem, for every $\epsilon>0$ we can find 
$\delta>0$ such that
  \begin{align}
    \label{eq:slow-approx}
    \lim_{n\to+\infty}\frac1n\sum_{i=0}^{n-1}-\log|f^i(x)|_\delta
    =
    \int_{-\delta}^{\delta}-\log|x|\,d\nu_0(x)<\epsilon
  \end{align}
  for every $x\in B(\mu)$, thus for Lebesgue almost every $x$.
\end{itemize}
Condition \eqref{eq:slow-approx} is known as a \emph{slow
  recurrence condition} on the singular set $\{0\}$ of
$f$. Under these conditions, together with the
non-degeneracy property given by
Remark~\ref{rmk:non-degenerate}, we have

\begin{lemma}\label{le:hyp-time2}
  There exists \( \theta>0 \) and
  $0<\delta<1$ (depending only on $f$ and the expanding
  rate $\sqrt2$) such that, for Lebesgue almost every \( x
  \in I \), we can find $n_0\ge1$ satisfying: for each
  $n>n_0$ there are $(\sigma,\delta)$-hyperbolic times
  $1 \le n_1 < \cdots < n_l \le n$ for \( x \) with
  $l\ge\theta n$.
\end{lemma}

\begin{proof}
  See Lemma 5.4 of \cite{ABV00}.  
  Let us remark that here we have $\sigma=1/\sqrt2\approx 0.707$.
\end{proof}

Notice that the constants $\alpha,\beta$ and the
lower bound for the expansion rate $\sqrt2$ vary slightly in a
$C^2$ neighborhood $\cU$ of the geometric Lorenz flow $X$.
Likewise, the value of $\delta_1$ from Lemma~\ref{le:hyp-time1} depends
continuously $\alpha,\beta$ and $\inf|Df|$, and so we can assume that
$\delta_1=\delta_1(Y)>\underline\delta_1$ for some uniform
constant $\underline\delta_1>0$ for all $Y\in\cU$.

\emph{Hence, from Remark~\ref{rmk:delta-dense} and
  Section~\ref{sec:persist-contract-fol}, we obtain a
  neighborhood $\cU$ of the geometric Lorenz flow $X$ in
  $\fX^2(M)$ such that, for all $Y\in\cU$, the set $O_N(Y)$
  of $N$-pre-images of $\{0\}$ under $f_Y$ is
  $\underline\delta_1/3$-dense in $I$ and does not contain
  the singular point.} Thus the point $0$ and the map $f$
satisfy all the conditions needed to perform the
construction of an induced uniformly expanding Markov map
$F$ from a neighborhood $\Delta=(-a,a)$ of $0$ to itself, as
presented in \cite{alves-luzzatto-pinheiro2004,gouezel}. More precisely,

\begin{theorem}\label{thm:markov}
  There exists a neighborhood $\Delta:=(-a,a)$, for some
  $0<a<1/2$, of the singular point $0$; a countable Lebesgue
  modulo zero partition $\cQ$ of \( \Delta \)
  into sub-intervals; a function \( R: \Delta
  \to \ZZ^+ \) defined almost everywhere, constant on
  elements of the partition \( \cQ \); and constants
  \( c >0, \kappa>1 \) such that, for all \(
  \omega\in\cQ \) and \( R=R(\omega) \),
  the map \( F:=f^{R}:\omega \to \Delta \)
  is a \( C^{2} \) diffeomorphism, satisfies the bounded
  distortion property and is uniformly expanding: for each
  \( x,y\in\omega \)
  \begin{align*}
    \left| \frac{Df^{R}(x)}{Df^{R}(y)} -1 \right| \le c
    |f^{R}(x)-f^{R}(y)|
    \quad\text{and}\quad
    |f^{R}(x)-f^{R}(y)|>\kappa|x-y|.
  \end{align*}
  Moreover, for each $\omega\in\cQ$ there exists $0<k\le N$
  such that $n:=R(\omega)-k$ is a
  $(\sigma,\delta_1)$-hyperbolic time for each
  $x\in\omega$; $\omega\subset V_{n}(x)$ and, in addition,
  $f^j(\omega)\subset I\setminus\Delta$ for all $n\le j <
  R(\omega)$.
\end{theorem}

It was proved in \cite{araujo2006a} that the one-dimensional
Lorenz transformation has \emph{exponentially slow
  recurrence} to the singular set, that is, for every
$\epsilon>0$ there exists $\delta>0$ such that
\begin{align}\label{eq:expslowrecurrence}
  \limsup_{n\to+\infty}\frac1n\log\lambda\left(\left\{x\in I:
    \frac1n \sum_{i=0}^{n-1}-\log|f^i(x))|_\delta>\epsilon\right\}\right)<0.
\end{align}
Following \cite{gouezel} this ensures the following result.

\begin{theorem}\label{thm:exp-tail} In the same setting of the
  previous Theorem~\ref{thm:markov}, the inducing time
  function \( R \) has exponential tail, that is, there
  exists positive constants $c$ and $\gamma$ such that
  $$
  \lambda(\{x\in I : R(x)>n\})<c e^{-\gamma n}.
  $$
\end{theorem}

This Markov map $F$ is obtained by inducing the interval map
$f$ on the interval $\Delta$, using an inducing time that is
given by the sum of a hyperbolic time with a non-negative
integer bounded by the number $N$ defined in
\eqref{rmk:delta-dense}, and has exponential tail with
respect to the Lebesgue measure.  Therefore, $R$ is Lebesgue
integrable and the following is well-known.

\begin{proposition}
\label{pr:existence-smooth-acim}
There exists an absolutely continuous invariant probability
measure $\nu$ for $F$ whose density $\phi=d\nu/d\lambda$ is
a $C^1$ strictly positive and bounded function on $\Delta$.
Moreover, $\nu_0= \frac{1}{\int R d\nu}
\cdot\sum_{k=0}^\infty (f^k)_*(\nu\mid R>k)$.
\end{proposition}

\subsubsection{Renyi condition and $C^1$ invariant density}
\label{sec:renyi-condit-c1}

  The previous induced map $F$ satisfies a stronger property
  than the expression above for bounded distortion, the
  \emph{Renyi condition} from \cite{Re57}. Indeed, by a
  simple computation
  $$
    \frac{|D^2F|}{|DF|^2}(x)  
    	\leq \sum_{j=0}^{R(x)-1} \frac{1}{|DF(x)|} 
        \frac{|D^2 f (f^j(x)|}{|Df (f^j(x)|}
  $$
  Therefore we can obtain the uniform upper bound
  \begin{align*}
    \frac{|D^2F|}{|DF|^2}(x)
    &\le
    \sum_{i=0}^{R-1}
    \frac1{|Df^{R-i}(f^i(x))|}\cdot
    \frac{|D^2f(f^i(x))|}{|Df(f^i(x))|^2}
    \lesssim
    \sum_{i=0}^{R-1} \sigma^{R-i}
    \frac{|f^i(x)|^{\alpha-2}}{(|f^i(x)|^{\alpha-1})^2}
    \\
    &=
    \sum_{i=0}^{R-1} \frac{\sigma^{R-i}}{|f^i(x)|^\alpha}
    \le
    \sum_{i=0}^{R-1} \frac{\sigma^{R-i}}{\sigma^{\alpha
        b(R-i)}}
    \le
    B=\frac{1-\alpha}{a\alpha}
    \sum_{i\ge0} \sigma^{(1-b\alpha)i}<\infty,
  \end{align*}
  for every $x\in \omega$ and $\omega\in\cQ$,
  where $R=R(\omega)$. This implies that, for $x,y\in\omega$
  \begin{align*}
    \left|\frac1{DF}(x)-\frac1{DF}(y)\right|
    \le
    |x-y|\frac{|D^2F|}{|DF|^2}(z)
    \le
    B|x-y|
  \end{align*}
  for some $z\in \omega$ given by the Mean Value
  Theorem. Moreover, given $n>1$
  \begin{align}
    \left|\frac{D^2F^n(x)}{(DF^n(x))^2} \right|
    &=
    \frac1{|DF^n(x)|^2}
    \left|\sum_{i=0}^{n-1} D^2F(F^i(x))
    \left(\prod_{j=0,\dots,n-1\atop j\neq i}DF(F^j(x))
    \right)\right| \nonumber
    \\
    &=
    \frac{1}{|DF^n(x)|}
    \sum_{i=0}^{n-1} \frac{|D^2F(F^i(x))|}{|DF(F^i(x))|}
    \le
    \frac{B}{|DF^n(x)|}
    \sum_{i=0}^{n-1} |DF(F^i(x))|  \nonumber
    \\
    &=
    B\sum_{i=0}^{n-1} \frac{|DF(F^i(x))|}{|DF^n(x)|}
    \le
    B \cdot n \cdot \sigma^{n-1}  \label{eq:conseq-renyi}
  \end{align}
  which is an infinitesimal when  $n\to+\infty$.

So, we get that there exists a $C^2$ uniformly expanding
Markov map $F$ satisfying the Renyi condition from the previous
remark.
In this setting the arguments of Baladi-Vall\'ee
from~\cite{BaVal2005} provide an invariant density for $F$
in the space of $C^1$ functions.

\begin{lemma}
  \label{le:smooth-dens}
  The density $\phi=d\nu/d\lambda$ of the $F$-invariant
  probability measure $\nu$ is a $C^1$ function
  $\phi:[-1/2,1/2]\to[0,+\infty)$.
\end{lemma}

\subsubsection{Uniform bounded distortion for powers of the induced map}
\label{sec:uniform-bounded-dist}

We need this technical result in the final arguments and we
are ready to prove it here.

\begin{proposition}
  \label{pr:bdd-dist-power}
  There exists $B_0>0$ such that for all $n\in\ZZ^+$,
  $h\in\cH_n$ and $x,y\in I$
  \begin{align*}
    \left|
      \frac{DF^n(h(x))}{DF^n(h(y))}
      \right|
      \le
      B_0 |x-y|.
  \end{align*}
\end{proposition}

\begin{proof}
  We use the properties of hyperbolic times and the last
  part of the statement of Theorem~\ref{thm:markov} to
  explicitly estimate this bound. First, we fix $n\in\ZZ^+$
  and set $R_i:=R(\omega_i)$ where $\omega_i=\cQ(F^i(h(y)))$
  is the element of the partition $\cQ$ containing both
  $F^i(h(y)), F^i(h(x))$, $i=0,\dots,n-1$. Secondly, we set
  $n_i:=R_i-k_i$ to be the hyperbolic time of $F^i(h(x))$
  given by Theorem~\ref{thm:markov}. Then we write $\bar
  x:=h(x), \bar y=h(y)$ and
  \begin{align*}
    \log\left|\frac{DF^n(\bar x)}{DF^n(\bar y)} \right|
    &=
    \sum_{i=0}^{n-1}\sum_{j=0}^{R_i-1}
    \log\left|\frac{Df(f^j(F^i(\bar x)))}{Df(f^j(F^i(\bar y)))}\right|
  \end{align*}
  where on hyperbolic times the summand can be bounded as follows
  \begin{align*}
    \log\left|\frac{Df(f^j(F^i(\bar x)))}{Df(f^j(F^i(\bar y)))}\right|
    &\approx
    \log\left|\frac{f^j(F^i(\bar x))}{f^j(F^i(\bar y))}\right|
    \le
    \frac{|f^j(F^i(\bar x))-
      f^j(F^i(\bar y))|}{|f^j(F^i(\bar y))|}
    \\
    &\le
    \frac{\sigma^{n_i-j}
      |f^{n_i}(F^i(\bar x))-f^{n_i}(F^i(\bar
      y))|}{\sigma^{b(n_i-j)}}
    \le
    \sigma^{b (R_i-j)} |F^{i+1}(\bar x)-F^{i+1}(\bar y)|
  \end{align*}
  since $f$ expands distances by at least $\sigma^{-1}$ on
  both branches. Putting this inequality in the summation
  above we get
  \begin{align*}
    \log\left|\frac{DF^n(\bar x)}{DF^n(\bar y)} \right|
    \le
    \frac{\sigma^b}{1-\sigma^b}\sum_{i=0}^{n-1}
    |F^{i+1}(\bar x)-F^{i+1}(\bar y)|.
  \end{align*}
  But $F$ also expands distances inside each partition
  element, thus
  \begin{align*}
    \log\left|\frac{DF^n(\bar x)}{DF^n(\bar y)} \right|
    \le
    \frac{\sigma^b}{1-\sigma^b}
    |x-y|\sum_{i=0}^{n-1}\kappa^{-(n-i)}
    \le
    \frac{\sigma^b}{1-\sigma^b}
    \frac{\kappa^{-1}}{1-\kappa^{-1}} |x-y|.
  \end{align*}
  To complete the proof it is enough to define
  $B_0:=\sigma^b(1-\sigma^b)^{-1} (\kappa-1)^{-1}$.
\end{proof}

\subsection{The good roof function}
\label{sec:good-roof-functi}

Here we show that the associated flow return time function
$r:\cup_{\omega\in\cQ}\omega\to(r_0,+\infty)$, where $r_0>0$
depends only on $f$, induced from $\tau$ and associated to
the induced map $F$, is a good roof function. 
Note that the function $r$ is defined from the Poincar\'e return time
function for every $x\in\bigcup_{\omega\in\cQ}\omega$ as
\begin{align*}
  r(x)=S_R\varrho(x):=\sum_{j=0}^{R(x)-1}
  \varrho(f^j(x)),
\end{align*}
where
$
  \varrho(x) := \inf\big\{ \tau(z) :
  z\in\pi^{-1}(\{x\})\big\}=\tau(x,0,1), \text{for } x\in
  I\setminus\{0\},
$
since $\tau$ does not depend of the point we choose on some
strong-stable leaf in $S^*$.
Clearly $r$ is still bounded from below by the same value
$\tau_0$ that bounds $\tau$ on each linearized flow box near
the origin; see Section~\ref{sec:near-singul}. This is
property (1) of a good roof function.

Furthermore, as a consequence of our construction in 
Section~\ref{sec:geometr-lorenz-flow}
and expressed in \eqref{eq:tau-pert}, there exists a
function $s_Y$, constant on strong-stable leafs, which is
$C^2$ close to a constant function on $S^*_Y$ such that for every $Y$ 
$C^2$-close to 
\begin{align}\label{eq:varrho-x}
  \varrho_Y(x)= -\frac1{\lambda_1(Y)} \log|x|+s_Y(x).
\end{align}
As a consequence we obtain that
\begin{align*}
  |D\varrho_Y(x)+ (\lambda_1(Y) x)^{-1}| = |D s_Y(x)|, x\in
  I\setminus\{0\}
\end{align*}
is uniformly $C^1$ close to the zero function. In
particular, we can find $\xi>0$ so that
$|Ds(x)|\le\xi|x|^{-1}$ and hence there are $\xi_1,\xi_2>0$
such that
\begin{align}\label{eq:Dvarrho-x}
  \xi_1\le - x\cdot D\varrho(x)\le \xi_2,
\end{align}
Notice also that one can take $|\xi_i-1/\lambda_1(X)|$ as small as
needed, for $i=1,2$, by taking $Y$ sufficiently $C^2$ close
to $X$.

\subsubsection{Exponential tail}
\label{sec:exponent-tail}

We split the estimates in three cases depending on how large the tail constant is. 
Consider the positive real 
$\xi=(2\nu(\rho))^{-1}=(2\int \rho \,d\nu)^{-1}$ and take a positive integer $L$.

\begin{description}
\item[Case $R$ big enough] If $R>\xi L$, then
  \begin{align}\label{eq:tail-1}
    \Leb\{r>L\,\&\, R>\xi L\}
    \le
    \Leb\{R>\xi L\}
    =
    \Leb\Big(\bigcup_{\omega\in\cP,\, R(\omega)>\xi L}
    \omega\Big)
    \le
    c e^{-\gamma \xi L}
  \end{align}
  since $R$ has exponential tail.
\item[Case $R$ not so big] If $R\le\xi L$, then $S_R\varrho
  > L$ implies that $S_R\varrho - \nu(\varrho) >
  L-\nu(\varrho)$ and
  \begin{align*}
    \frac1R\sum_{i=0}^{R-1}(\varrho \circ f^i -
    \nu(\varrho )) > \frac{L-\nu(\varrho)}{R}
    >
    \frac{L}R -\nu(\varrho )
    >
    \frac1\xi -\nu(\varrho)
    =
    \nu(\varrho)
    >0
  \end{align*}
  and also $\nu(\varrho\circ f^i -\nu(\varrho))
  =\int\Big( \varrho \circ f^i -\int \varrho \,
  d\nu\Big)\,d\nu=0$.
\end{description}
At this point we recall a large deviations result for
non-uniformly expanding maps (see e.g. 
\cite{araujo-pacifico2006,araujo2006a,Va09})
which guarantees that 
\begin{align*}
  \limsup_{n\to+\infty}\frac1n\log\Leb
  \left\{x: \frac1n \sum_{i=0}^{n-1} \varrho \circ f^i(x) > 2\nu(\varrho) >0 \right\} < 0
\end{align*}
so that the measure $\Leb\{x:\frac1R\sum_{i=0}^{R-1}\varrho\circ
f^i(x) >2\nu(\varrho) \}$ is exponentially small in $R$ (thus in $\xi L$
also), but only for $R\ge R_0$, for some integer $R_0$.  We
remark that $R_0$ does not depend on the value of $L$.  Thus
we have achieved an exponential tail for
\begin{align}\label{eq:tail-2}
  \Leb\{ x\in\Delta : R_0\le R(x) \le \xi L \,\&\, r>L\}.
\end{align}
\begin{description}
\item[Case $R$ small ($R<R_0$)] It is enough to consider the
  case $L\gg R_0$ since we are only interested on the
  measure of the tail set of $r$. Hence $r>L\iff
  \sum_{i=0}^{R-1}\varrho\circ f^i>L$ implies
  $\varrho\circ f^i > L/R_0$ for some $i\in\{0,\dots,
  R_0\}.$ Hence
  \begin{align*}
    \Leb\{r>L \,\&\, R\le R_0\} \le 
    \Leb\{ \varrho \circ f^i>L/R_0\}
    = (f_*^i\Leb)\big\{\varrho >L/R_0\}.
  \end{align*}
  Since $|Df|>\sqrt2$, $f$ has two $C^2$ monotonous branches
  and $0\le i\le R_0$, we have that the density of
  $f^i_*\Leb$ is smaller than $2^{i/2}$: for each branch $f\mid
  I^\pm$ we have $(f\mid I^\pm)_*\Leb$ with density smaller
  than $2^{-1/2}$; $f$ is not Markov and $f^i$ has $2^i$
  branches whose images might intersect, so that the maximum
  density would be, in the worst case where the image of
  every branch intersects at some region, smaller than
  $2^i\cdot 2^{-i/2}=2^{i/2}$. We conclude
  \begin{align*}
    \Leb\{r>L \,\&\, R\le R_0\}
    &\le 
    2^{i/2}\Leb\{\varrho>L/R_0)
    \\
    &\le
    2^{R_0}\Leb\{\varrho>L/R_0\}
    \le
    2^{R_0}\cdot e^{-\lambda_1 L/R_0}.
  \end{align*}
  At this point we use \eqref{eq:varrho-x} and conclude that
  the measure $\Leb\{r>L \,\&\, R\le R_0\}$ decays
  exponentially fast with $L$.
\end{description}
This estimate together with \eqref{eq:tail-1} and
\eqref{eq:tail-2} shows that $r$ has exponential tail.

\subsubsection{The uniform bound of the derivative}
For property (2) of a good roof function, let $h\in\cH_1$,
$h:\Delta\to\omega$ be an inverse branch of $F=f^R$ with
inducing time $l=R(\omega)\ge1$ and let us fix
$x\in\omega$. Then
\begin{align*}
  |D(r\circ h)(x)|
  &=
  |Dr ( h(x) )|\cdot |Dh(x)|
  =
  \frac{|Dr ( h(x) )|}{|DF ( h(x) ) |}
  =
  \left|
    \sum_{i=0}^{l-1} \frac{(D\varrho\circ f^i)\cdot Df^i }{DF}\circ h(x)
    \right|.
\end{align*}

In addition, from the construction of the inducing partition
using hyperbolic times, we have $l=(l-n) + n$, where $l-n$
is a $(\sigma,\delta_1)$-hyperbolic time for $x_0$ and
$0<n\le N$.
Thus $|x_i|\ge \sigma^{b(l-i)}$ for $x_i=f^i(h(x))$ by
definition of hyperbolic times, and so, by
\eqref{eq:Dvarrho-x} we get $|D\varrho(x_i)|\le \xi_2
\sigma^{-b(l-i)}$, $i=0,\dots,l-n$, where
$\sigma=1/\sqrt2$. Moreover
\begin{align*}
  \left|\frac{Df^i}{DF}\right|\circ h(x) 
  &=
  \frac1{|Df^{l-i}\circ f^i|}\circ h(x) 
  \le
  \sigma^{(l-i)},\quad i=0,\dots,l-1;
\end{align*}
and $|x_i|\ge\delta$ and $Df(x_i)>\sigma^{-1}$ for
$l-n<i<R(\omega)$.  Altogether this implies, because
$0<b<1/2$, that $ |D(r\circ h)(x)| \le
(\xi_2/\delta)\sum_{i=0}^{l-1} \sigma^{(1-b)i}$ which is
bounded by a constant. Thus we have proved that
$\sup_{h\in\cH_1}\| D(r\circ h)\|_0$ is finite.

\subsubsection{Uniform non-integrability}
\label{sec:uniform-non-integr}

We prove that $r$ satisfies the aperiodicity condition, the
third item in Definition~\ref{def:goodroofunc}, for a $C^k$
open set of vector fields, $k\ge2$, throught a perturbative
argument.

If there exists a $C^1$ function $u:\Delta\to\RR$ and a
measurable function $v:\Delta\to\RR$ constant on each
element $\omega$ of $\cQ$ satisfying $ r =u \circ F -u+v$,
then we show that, up to a $C^1$ perturbation of the original 
vector field $X$, 
this is impossible for all $C^2$ nearby vector fields. For
this we choose two distinct periodic points $x_{1}$, $x_{2}$
for $F:\De \circlearrowleft$ of the same period $n$ whose
orbits are (i) distinct, and (ii) each orbit visits each of
the elements of the Markov partition the same number of
times as the other, but (iii) necessarily in some different
order to each other. The existence of such a pair of
periodic orbits is a consequence of $F$ being a full branch
Markov map: if $\omega_1,\omega_2$ are two elements of the
Markov partition, we can choose the period $p=4$ and $x_i,
i=1,2$ such that
\begin{align*}
  x_1, F(x_1) \in \omega_1, F^2(x_1), F^3(x_1) \in
  \omega_2 \quad\text{and}\quad 
x_2, F^2(x_2) \in \omega_1, F(x_2), F^3(x_2) \in
  \omega_2.
\end{align*}
Furthermore, $x_1, x_2$ can be chosen in the interior
of $\omega_1$.
The cohomological equation implies
$S_{p}r(x_1)=S_pv(x_1)=S_pv(x_2)=S_pr(x_2)$ since $x_1$ and
$x_2$ visit the same Markov partition elements an equal
number of times and $v$ is constant on each partition
element. Hence it is enough to modify the roof function
$\rho:S\to\RR^+$ in a small neighbourhood of $x_1$ that does
not intersect the orbit of $x_2$ to ensure that the induced
roof function $r$ satisfies $S_{p}r(x_1)>S_{p}r(x_2)$ and so
is not cohomologous to a piecewise constant roof
function. This modification can be done by changing the size
of the vector field $X$ in a small neighborhood around
$x_1$. Note that $S_pr(x_i)$ is the period of $x_i$ as a
periodic orbit of the vector field $X$. 
By some abuse of notation we shall denote by $X$ the perturbed vector field.

We need to show that this conclusion holds for all vector
fields $Y$ that are $C^k$ close to $X$, $k\ge2$.
The orbits of these points involve only finitely many
iterates (the inducing time $R$) of the Lorenz
transformations $f$: both orbits are contained in $\{R\le
N\}$ for some fixed $N\ge1$. Moreover, these points belong
to two distinct hyperbolic periodic orbits (of saddle type)
of the geometric Lorenz attractor and are away from a
neighborhood of the singularity at the origin. Hence, these
orbits admit a smooth continuation to all $C^k$ nearby
vector fields $Y$ which admit a similar construction of
smooth cross-section $S_Y$ and induced
tranformation $F_Y$, following the inductive procedure
detailed in \cite{alves-luzzatto-pinheiro2004}.

In particular, since the periodic points $x_1, x_2$ belong to the interior
of the partition elements we can control finitely many iterates of the
Lorenz transformation $f_Y$ and obtain that the
corresponding partition of $\{R_Y\le N\}$ associated to the
induced transformation $\tilde F$ up to inducing time $N$ is
close enough to the partition of $\{R\le N\}$ so that the
continuation $\tilde x_i$ of the orbits of $x_i, i=1,2$ have
the same combinatorics as before, visiting the same elements
of the Markov partition the same number of times as the
other but in a different order, and with the same inducing
times.  The rest of the inducing map for $f_Y$ is obtained
following an inductive construction and we do not use in
this argument the elements of the partition whose induction
time is higher than $n$.

We can then use the same expression as before obtaining $S_p\tilde
r(\tilde x_1)>S_p\tilde r(\tilde x_2)$ for $Y$ sufficiently
$C^k$ close to $X$, showing that the
induced roof function $\tilde r$ for the vector field $Y$
cannot be cohomologous to a piecewise constant roof
function. Hence this function satisfies the uniform
non-integrability condition needed to obtain exponential
decay of correlations for the flow of $Y$ on the geometric
Lorenz attractor.


\subsection{The hyperbolic skew-product structure}
\label{sec:skew-product-struct}

Now we explain how the existence of the previously
constructed induced map $F$, together with the existence of
the contracting foliation on the cross-section $S$ of the
geometric Lorenz flow, ensures the existence of the good
hyperbolic skew-product structure for the flow.

We start with a useful consequence of Lemma~\ref{le:leo} and
Remark~\ref{rmk:densepreorbit}: the images of $\Delta$ cover
$I$ with the exception of a set of points of zero Lebesgue
measure, i.e.
\begin{align}\label{eq:Deltacobre}
  \bigcup_{\omega\in\cQ}\bigcup_{j=0}^{R(\omega)-1}
  f^j(\omega)=I,\quad\lambda\bmod0.
\end{align}
In fact, we have that $\Delta_\cQ:=\cup_{\omega\in\cQ}\omega =
\Delta\setminus N$ with $\lambda(N)=0$ by construction, and
every point of the domain of $f$ has some pre-image in
$\Delta$. Since $f^{R(\omega)}(\omega)=\Delta$ for each
$\omega\in\cQ$ we have that
\begin{align*}
  (-1/2,1/2)\subset \bigcup_{j\ge0} f^j(\Delta) =
  \bigcup_{j\ge0} f^j(N) \cup
  \bigcup_{\omega\in\cQ}\bigcup_{j=0}^{R(\omega)-1}
  f^j(\omega)
\end{align*}
and $\lambda(f^j(N))=0$ for all $j\ge1$ because $f$ is
piecewise $C^2$, which proves~\eqref{eq:Deltacobre}.

\subsubsection{The induced Poincar\'e return map}
\label{sec:induced-poincare-ret}

We define the following induced map $\widehat
F:=P^{R\circ\pi}:\pi^{-1}(\Delta_\cQ)\to\pi^{-1}(\Delta)$,
where $\pi:S\to I$ is the projection onto the quotient
$I=S/\cF$ of $S$ over the stable leaves.
We set for future use $\widehat\Delta=\pi^{-1}(\Delta)$ and 
$\widehat\Delta_\cQ:=\pi^{-1}(\Delta_\cQ)=\cup_{\omega\in\cQ}\pi^{-1}(\omega)$.

We note that on each element $\pi^{-1}(\omega)$ of the
$\Leb\bmod0$ partition $\hat \Delta_{\mathcal Q}$
we have an inducing time given by $R(\omega)$ (we write
$\Leb$ for a normalized area measure on $S$). Moreover by construction
$\pi\circ\widehat F=F\circ\pi$. The uniform
contraction along the leaves of $\cF$ ensures that $\widehat F$
contracts distances between points in the same leaf. 
This shows that properties (1) and (4) in 
Definition~\ref{def:hyp-skew-product} hold for $\widehat F$ on $\widehat \Delta$.
Properties (2) and (3) will be proven in Subsections~\ref{sec:hat-f-invari} and 
\ref{sec:disint-property}, respectively.

\subsubsection{The existence of smooth conjugation}
\label{sec:existence-smooth-sem}

Here we explain how to obtain the smooth semi-conjugacy between the original
geometric Lorenz flow model and the hyperbolic skew-product
model.

We can define the semiflow $\widehat F_t$ over
$\widehat\Delta$ with base map $\widehat F$ and height
function $r\circ\pi$ as usual, whose phase space is
$\widehat\Delta_r$ defined in
Section~\ref{sec:good-hyperb-semifl}.
The suspension semiflow $Z_t$ with base map $P:S^*\to S$
and phase space
\begin{align*}
  \Delta_\tau:= \{ (w,t) : w\in S^*, 0\le t
  <\tau(w)\}
\end{align*}
can easily seen to be conjugated to the geometric Lorenz
flow $Y^t$ in a neighborhood of $\Lambda$ through the smooth
transformation $\Phi:\Delta_\tau\to U$ given by
$\Phi(w,t)=Y^t(w)$, where $U$ is a neighborhood of the
attractor $\Lambda$ defined in \eqref{eq:trap-reg-U}. This
is a diffeomorphism on an open subset of $U$ with full
volume on $M$.

We need to conjugate $Z_t$ with $\widehat F_t$. Since the
first return time function $\tau$ is constant on
strong-stable leaves, the main diferences between $Z_t$ and
$\widehat F_t$ are the base maps and the roof functions.
But $F:\Delta_\cQ\to\Delta$ and $r$ are induced from $f$ and
$\varrho$ with the same number of interates on each
$\omega\in\cQ$. In fact, since the roof function $r$
associated to $\widehat F_t$ is an ergodic sum of the roof
function $\varrho$ associated to $Z_t$ with a locally
constant number of summands, which precisely equals
$R(\omega)$ on each $\omega\in\cQ$, it follows, by the
definition of the equivalence relation $\sim$ defining
$\Delta_\tau$:
\begin{enumerate}
\item[(i)] $\widehat\Delta_r$ can be naturally identified
  with an open subset of $\Delta_\tau$;
\item [(ii)]
  from~\eqref{eq:Deltacobre}, $\widehat\Delta_r$ has in fact
  full bidimensional Lebesgue measure in
  $\Delta_\tau$; and
\item [(iii)] $\widehat
  F_t(w,s)=Z_t(w,s)$ for all
  $(w,s)\in\widehat\Delta_r$ and $t\ge0$.
\end{enumerate}

This shows that we can smoothly conjugate $\widehat F_t$
with $Z_t$ over an open subset with full Lebesgue measure;
then smoothly conjugate $Z_t$ with the original geometric
Lorenz flow $Y^t$ on an open subset with full volume in a
neighborhood of the attractor $\Lambda$.

\subsubsection{The $\widehat F$-invariant probability}
\label{sec:hat-f-invari}

In this subsection we use that every invariant measure
associated to a quotient map over a stable foliation lifts
in a unique way to an invariant measure for the original
dynamics to prove item (2) of
Definition~\ref{def:hyp-skew-product}.

Let $(S,d)$ be a compact metric space, $\Gamma\subset S$ and
let $P: S\setminus\Gamma \to S$ be a measurable map.  We
assume that there exists a partition $\cF$ of $S$ into
measurable subsets, having $\Gamma$ as at most countable
collection of elements of $\cF$, which is
\begin{itemize}
\item{{\em invariant:\/}} the image by $P$ of any $\xi\in\cF$ distinct
     from $\Gamma$ is contained in some element $\eta$ of $\cF$;
\item{{\em contracting:\/}} the diameter of $P^n(\xi)$ goes to
     zero when $n\to\infty$, uniformly over all the $\xi\in\cF$
     for which $P^n(\xi)$ is defined.
\end{itemize}
Set $\pi:S\to \cF$ to be the canonical projection.  Hence,
$A\subset \cF$ is Borel measurable if and only if
$\pi^{-1}(A)$ is a Borel measurable subset of $S$, since $A$
is open if, and only if, $\pi^{-1}(A)$ is open in $S$.  The
invariance condition ensures that there is a uniquely
defined map
$$
f:(\cF\setminus\{\Gamma\}) \to \cF
\quad\text{such that}\quad
f\circ \pi = \pi \circ P,
$$
which is measurable.  We assume that the leaves are
sufficiently regular so that $\Xi/\cF$ is a metric space
with the topology induced by $\pi$.

Let $\mu_f$ be any probability measure on $\cF$ invariant
under the transformation $f$.

For any bounded function $\psi:S\to\RR$, let
$\psi_{-}:\cF\to\RR$ and $\psi_{+}: \cF\to\RR$ be
defined by
$$
\psi_{-}(\xi)=\inf_{x\in\xi}\psi(x)
\qquad\mbox{and}\qquad
\psi_{+}(\xi)=\sup_{x\in\xi}\psi(x).
$$

\begin{proposition}
  \label{pr:liftmesureP}
  There exists a unique measure $\mu_P$ on $S$
such that
$$
\int \psi\,d\mu_P
=\lim_{n\to\infty} \int (\psi\circ P^n)_{-}\,d\mu_f
=\lim_{n\to\infty} \int (\psi\circ P^n)_{+}\,d\mu_f
$$
for every continuous function $\psi:S\to\RR$.  Besides,
$\mu_P$ is invariant under $P$. Moreover the correspondence
$\mu_f\mapsto\mu_P$ is injective, 
$\pi_*\mu_P=\mu_f$ and $\mu_P$ is
ergodic if $\mu_f$ is ergodic.
\end{proposition}

This follows from standard arguments which can be found in,
e.g.  Section 7.3.5 of \cite{AraPac2010}.  Hence we just
have to take $\mu_f=\nu_0$ to obtain the corresponding
$\eta_0=\mu_P$ ergodic $P$-invariant probability measure
which lifts $\nu_0$, where $P$ and $f$ are the Poincar\'e
return map to the cross-section $S$ of the geometric Lorenz
flow, $f$ the Lorenz transformation associated to $P$;
and $\cF$ is the family of stable leaves on $S$ for $P$.

Analogously, we consider the measurable map $\widehat
F:\widehat\Delta_\cQ\to\widehat\Delta$ on the
space $\widehat\Delta$ with the same foliation $\cF$ of
$S$ restricted to $\widehat\Delta$, together with the
quotient map $F:\Delta_\cQ\to\Delta$. Then we start
with the $F$-invariant ergodic measure $\nu$ and obtain
an $\widehat F$-invariant ergodic measure $\eta$ on
$\widehat\Delta$.

\subsection{The disintegration property}
\label{sec:disint-property}

Here we show that the previous measure $\eta$ admits a
smooth disintegration as stated at item (3) of
Definition~\ref{def:hyp-skew-product} under the following
assumptions on $\cF$ and $f$, besides invariance and
contraction as in the previous subsection:
\begin{itemize}
\item $S/\cF$ is the compact closure of an open domain of a
  finite dimensional smooth manifold;
\item $f:\cF\setminus\Gamma\to \cF$ is a 
  uniformly expanding
  Markov map, according to
  Section~\ref{sec:unif-exp-markov};
\item The invariant density
  $\phi=d\mu_f/d\Leb$ is a $C^1$ function.
\end{itemize}
We note that the assumption of denumerability of $\Gamma$
ensures that $S\setminus\Gamma$ is $\sigma$-compact.

The general strategy of the argument is to obtain the
disintegration of $\eta$ as fixed point of a certain
transfer operator whose action from fiber to fiber varies
differentiably. 

To avoid the introduction of extra notation and to focus on
the the geometric Lorenz attractor case, from now on we take
$\widehat F:\widehat\Delta_\cQ\to\widehat\Delta$ on the
space $\widehat\Delta$ with the foliation $\cF$ of $S$
restricted to $\widehat\Delta$, together with the quotient
map $F:\Delta_\cQ\to\Delta$, as our main maps.
Let us consider the set $\Omega$ of measurable families of
probability measures $\omega=(\omega_x)_{x\in I}$ supported
on the strong-stable leaves
$\{\pi^{-1}(x)\cap\Lambda\}_{x\in I}$ inside the geometric
Lorenz attractor $\Lambda$. We note that $I\ni x\mapsto
\omega_x$ is measurable if the real function
$x\mapsto\int\psi\,d\omega_x$ is measurable for every
continuous function $\psi:\widehat\Delta_\cQ\to\RR$ with
compact support.  Each such family defines a probability
measure $\widetilde{\nu^\omega}$ through the linear
functional
  \begin{align*}
  C^0_0(\widehat\Delta_\cQ,\RR)\ni \psi \mapsto \int \int \psi \,
  d\omega_x \,d\nu(x)    
  \end{align*}
  where $C^0_0(\widehat\Delta_\cQ,\RR)$ is the set of all
  continuous functions on $\widehat\Delta_\cQ$ with compact
  support; see e.g. \cite{EG92} for the definition and
  properties of Radon measures.
We define the operator $\cL: \Omega \to \Omega$ such that
\begin{align*}
  \int \int\psi\,d\cL(\omega)_x \,d\nu(x)
  =
  \int \int \psi\circ \widehat F \,d\omega_x \,d\nu(x),
  \quad
  \psi\in C^0_0(\widehat\Delta_\cQ,\RR).
\end{align*}
This is the dual of the usual Koopman operator $U\colon
C^0_0(\widehat\Delta_\cQ,\RR) \to
C^0_0(\widehat\Delta_\cQ,\RR)$ given by $U(\psi)=\psi \circ
\widehat F$.
Moreover, it defines an operator using the
disintegration of the measure $\widetilde{\nu^\omega}$
defined by the linear functional
\begin{align*}
  C^0_0(\widehat\Delta_\cQ,\RR)\ni \psi \mapsto \int \int
  \psi\circ \widehat F \,
  d\omega_x \,d\nu(x)
\end{align*}
with respect to the measurable partition of $\widehat\Delta$
given by the restriction of $\cF$ to this set. This defines
the family $\big(\cL(\omega)\big)_x$ for $\nu$-almost every
$x\in I$ in a unique way; see e.g. \cite{Ro62} for more on
disintegration with respect to measurable partitions of
Lebesgue spaces.
The next lemma establishes that invariant measures arise as fixed
points for $\cL$. More precisely,

\begin{lemma}
\label{le.fconvergence}
Given $\omega\in\Omega$ and $\psi\in
C_0^0(\widehat\Delta_\cQ,\RR)$ there exists the limit
\begin{align}
  \label{eq:bothlim}
\nu^{\bar\omega}(\psi):=\lim_n
(\cL^n\omega)\psi=\lim_n\int\int\psi\circ \widehat F^n\,d\omega_x \,d\nu(x).
\end{align}
Moreover, the probability measure $\nu^{\bar\omega}$ is $\widehat F$-invariant and the
  family $\bar\omega$ does not depend on $\omega$ on
  $\nu$-almost every point; in fact $\nu^{\bar\omega}=\eta$.
\end{lemma}

\begin{proof}
  Let $\psi\in C^0_0(\widehat \Delta_\cQ,\mathbb R)$ and $\omega\in\Omega$ be fixed.  Given
  $\epsilon>0$, let $\delta>0$ be such that
  $|\psi(A)-\psi(B)|\le\epsilon$ for all $A,B\subset S$ with
  $\dist(A,B)\le \delta$, where $\dist$ denotes the
  euclidean distance.  Since the partition $\cF$ is assumed
  to be contracting, there exists $n_0\ge 0$ such that
  $\diam(\widehat F^n(\xi))\le\delta$ for every $\xi\in\cF$ and any
  $n\ge n_0$.  Let $n + k \ge n \ge n_0$, $A\subset\pi^{-1}\{x\}$
  and $B\subset \pi^{-1}\{f^k(x)\}$. Then
  \begin{align}\label{eq:unif-cont}
    |\psi\circ \widehat F^{n+k}(A) - \psi\circ \widehat F^n(B)|
    \le
    \sup (\psi \mid \widehat F^{n+k}(\pi^{-1}\{x\} ))
    - \inf (\psi \mid \widehat F^{n}(\pi^{-1}\{F^k(x)\}))
    \le \epsilon
  \end{align}
  since $\widehat F^{n+k}(\pi^{-1}\{x\}) \subset  \widehat F^n(\pi^{-1}\{F^k(x)\})$.
Thus, using the $F^k$-invariance of $\nu$ we get from the previous estimate that
\begin{align*}
  |(\cL^{n+k}\omega)\psi-(\cL^n\omega)\psi|
  &=
  \left|\int\int \psi \circ \widehat F^{n+k} \,d\omega_{x}\,d\nu(x) - 
    \int\int \psi \circ \widehat F^{n}\,d\omega_{x}\,d\nu(x)\right|
  \\
  &=
  \left|\int\int \psi \circ \widehat F^{n+k} \,d\omega_{x}\,d\nu(x) - 
    \int\int \psi \circ \widehat
    F^{n}\,d\omega_{F^k(x)}\,d\nu(x)\right|
  \\
  &\le
  \int\left|
    \int \psi \circ \widehat F^{n+k} \,d\omega_{x}
    -\int \psi \circ \widehat
    F^{n}\,d\omega_{F^k(x)}\right|d\nu(x)
  \le\epsilon.
\end{align*}
This shows that the sequence $(\cL^n\omega)\psi$
is a Cauchy sequence and so it converges in the Banach space
$C^0_0(\widehat \Delta_\cQ,\mathbb R)$. It is
straightforward to check that, since each element of the
sequence is a normalized positive linear functional, the
limit is a family $\bar\omega\in\Omega$ such that
$\nu^{\bar\omega}$ has the same functional properties, and
so represents a probability measure.
We remark that for any compact subset $K$ of
$\widehat\Delta$ we can rewrite the last inequality above as
follows
\begin{align}\label{eq:unifcomp}
  \int\left|
    \int_K \psi \circ \widehat F^{n+k} \,d\omega_{x}
    -\int_K \psi \circ \widehat
    F^{n}\,d\omega_{F^k(x)}\right|d\nu(x)
  \le\epsilon.
\end{align}
This shows that the convergence does not depend on the
support of $\psi\circ\widehat F^n$ for arbitrarily large
$n$.
Note also that given $\vfi\in C^0_0(\Delta_\cQ,\RR)$ the
function $\vfi\circ\pi$ belongs to
$C_0^0(\widehat\Delta_\cQ,\RR)$ and is constant on each leaf
of $\cF$ through $\widehat\Delta$, which yields
\begin{align*}
  \nu^{\bar\omega}(\vfi\circ\pi)
  =
  \lim_n\int\int\vfi\circ\pi\circ \widehat F^n \, d\omega_x\,
  d\nu(x)
  =
  \lim_n\int \vfi(F^n(x)) \, d\nu(x)
  =
  \int\vfi\,d\nu,
\end{align*}
that is $\pi_*(\nu^{\bar\omega})=\nu$. 
Hence, if we show that $\nu^{\bar\omega}$ is $\widehat
F$-invariant, we can use Proposition~\ref{pr:liftmesureP} to
conclude that $\eta=\nu^{\bar\omega}$ independently of the
starting family $\omega\in\Omega$.  To prove invariance, we
observe that
\begin{align*}
  (\cL^{n+1}\omega)\psi
  =
  \cL(\cL^n\omega)\psi
  =
  \int\int\psi\circ \widehat F\,d(\cL^n(\omega)_x)\,d\nu(x)
\end{align*}
but $\psi\circ\widehat F$ is \emph{not continuous} with
compact support for $\psi\in
C_0^0(\widehat\Delta_\cQ,\RR)$. From~\eqref{eq:unifcomp} we
have that, using the $\sigma$-compactness of
$\widehat\Delta_\cQ$ and choosing a nested increasing
sequence $K_l$ of compact sets growing to $\widehat\Delta$
and a \emph{non-negative} $\psi\in
C^0_0(\widehat\Delta_\cQ,\RR)$, we get for $n\ge n_0, m\ge1,
l\ge1$
\begin{align}\label{eq:unifineq}
  \left|
    \int\int_{K_l}\psi\circ\widehat F\,d(\cL^n\omega)_x\,d\nu(x)
    -
    \int\int_{K_l}\psi\circ\widehat F\,d(\cL^{n+m}\omega)_x\,d\nu(x)
    \right|\le\epsilon.
  \end{align}
To ensure that $\psi\circ\widehat F\mid K_l$ is continuous with
compact support, we observe that
\begin{align*}
  \supp(\psi\circ\widehat F)
  =
  \widehat F^{-1}(\supp\psi)
  =
  \bigcup_{\omega\in\cQ}\big(P^{R(\omega)}\big)^{-1}(\supp\psi)
\end{align*}
is a denumerable union of compacts in $\widehat\Delta$,
because $\supp\psi$ is compact and $P:S^*\to S$ is a
diffeomorphism onto its image. Thus we can choose an
enumeration $\{\omega_n\}_{n\ge1}$ of $\cQ$ and define
\begin{align*}
  K_l:=\bigcup_{i=1}^l \big(P^{R(\omega_i)}\big)^{-1}(\supp\psi)
\end{align*}
to obtain a sequence such that $\psi\circ\widehat F\mid
K_l\nearrow\psi\circ\widehat F$ is a monotonous sequence of
continuous functions of compact support.
Hence, letting $m$ grow without bound in~\eqref{eq:unifineq}
we arrive at
\begin{align*}
  \left|
    \int\int_{K_l} \psi\circ\widehat F
    \,d(\cL^n\omega)_x\,d\nu(x)
    -
    \int\int_{K_l} \psi\circ\widehat F
    \,d\bar\omega_x\,d\nu(x)
    \right|\le\epsilon.
\end{align*}
Making $l$ grow we finally obtain
$|\cL^{n+1}\psi-\cL(\bar\omega)\psi|\le\epsilon$ for $n\ge
n_0$. These arguments assumed that $\psi$ is non-negative;
but for a continuous function with compact support, we can
write $\psi=\psi^+-\psi^-$ with $\psi^\pm$ non-negative and
still continuous, and then apply the same argument to each
summand using linearity.

Finally, $\epsilon>0$ can be arbitrarily chosen at the
start, so we have proved that
\begin{align*}
  \bar\omega=\lim_n \cL^n\omega=\cL(\bar\omega)
  \quad\text{which implies}\quad
  \nu^{\bar\omega}=\widehat F_*(\nu^{\bar\omega})
\end{align*}
as needed to complete the uniqueness part of the statement.
\end{proof}

\subsubsection{Conditional measures as uniform limits}
\label{sec:condit-measur-as}

To prove that $\bar u(x):=\int u(x,y)\,d\bar\omega_x(y)$ is
a $C^1$ map with bounded derivative, for each $C^1$ function
$u:\widehat\Delta_\cQ\to\RR$ with compact support, we need
some preliminary results.
%
%
%
We can be more precise about the operator $\cL$ in the next proposition,
whose proof will be given later in this section.

\begin{proposition}
  \label{pr:cLdisint}
  For every bounded measurable function
  $\psi:\widehat\Delta_\cQ\to\RR$ and for the constant
  family $\upsilon=(\upsilon_x)\in\Omega$ with
  $\upsilon_x=\lambda\mid_{\pi^{-1}(\{x\}) }$ for each $x\in
  I$, we have for $\lambda$-almost every $x\in I$
  \begin{align*}
    \int\psi\,d(\cL^n\upsilon)_x
    =
    \sum_{h\in\cH_n}
    \int \frac{\phi\cdot\psi\circ \widehat F^n_t}{\phi\circ
      F^n\cdot DF^n} (h(x)) \, d\lambda(t), \quad 
    \text{for each} \quad n\ge1
  \end{align*}
  where $\phi:=d\nu/d\lambda$ is the H\"older-continuous
  density of $\nu$ with respect to $\lambda$; and we write,
  to simplify the notation, $\widehat F^n_t(z):=\widehat
  F^n(z,t)$ for $(z,t)\in \widehat\Delta_\cQ$.
\end{proposition}

\begin{remark}
  \label{rmk:RPF}
  Note that $DF>0$ and so it is useful to write
  the transfer or Ruelle-Perron-Frobenius operator
  associated to $F$ and the potential $-\log|DF|$ as
  $$
  \cP(\vfi)(x):=\sum_{h\in\cH_1} \frac1{DF(h(x))}  \vfi(h(x)).
  $$ 
  Hence,
  $\cP \phi=\phi$ and 
  \begin{align*}
    \sum_{h\in\cH_n}
    \frac{\phi\cdot\psi\circ \widehat F^n_t}
    {\phi\circ  F^n\cdot DF^n} \circ h
    =
    \frac1{\phi}
        \sum_{h\in\cH_n}
    \frac{\phi\cdot\psi\circ \widehat F^n_t}
    {DF^n} \circ h
    =
      \frac1{\phi}\cP^n(\phi\cdot\psi\circ \widehat F^n_t).
  \end{align*}
\end{remark}

From Lemma~\ref{le.fconvergence} we can obtain the invariant
family $\bar\omega$ as the limit of
$(\cL^n\upsilon)_{n\ge1}$, so Proposition~\ref{pr:cLdisint}
provides an explicit expression to approximate the
elements of $\bar\omega$. The proof of this proposition
becomes simpler if we use the following lemmas.

\begin{lemma}
  \label{le:abs-conv}
  For every fixed $n\in\ZZ^+$, every $t\in [-\frac12,\frac12]$ and each bounded measurable
  function $\psi:\widehat\Delta_\cQ\to\RR$, the series given
  by $\frac1{\phi}\cP^n(\phi\cdot\psi\circ \widehat F^n_t)$
  is absolutely convergent Lebesgue almost everywhere.
\end{lemma}

\begin{proof}
  Indeed, since $\phi$ is a $C^1$
  function bounded from above and below (see
  Subsection~\ref{sec:renyi-condit-c1} and
  Lemma~\ref{le:smooth-dens}) there exists $C>0$, depending only on
  $\phi$, such that for $\lambda$-almost
  every $x\in I$
  \begin{align*}
    \left| \sum_{h\in\cH_n}
    \frac{\phi\cdot\psi\circ \widehat F^n_t}{\phi\circ
      F^n\cdot DF^n} (h(x))
      \right|
      \le
      \sum_{h\in\cH_n}\frac{ C \cdot\|\psi\|_\infty}{DF^n(h(x))}
      \le
      \sum_{h\in\cH_n}\frac{ C \cdot
        \lambda(h(\Delta))}{\lambda(\Delta)}
      <\infty
  \end{align*}
  since $\psi$ is essentially bounded,
  $\{h(\Delta)\}_{h\in\cH_n}=\cQ$ is a partition of $\Delta$
  Lebesgue modulo zero, and by the bounded distortion
  property combined with the mean value theorem. Indeed, from
  Proposition~\ref{pr:bdd-dist-power} we have for every
  $h\in\cH_n, x\in\Delta$ and some $z=z(h)\in h(\Delta)$
  \begin{align}\label{eq:bdd-unif-Delta}
    \lambda(\Delta)=\lambda(h(\Delta))DF^n(h(z))
    \le
    B_0\cdot \lambda(h(\Delta))DF^n(h(x))
  \end{align}
  and $B_0$ does not depend on $n\ge1$.
\end{proof}

From the previous argument we obtain a useful property for the transfer operator
$\cP$ for the expanding map $F$, to be used in what follows.

\begin{lemma}
  \label{le:spectral1}
  The spectral radius of the operator $\cP:L^\infty(I,\lambda)\to
  L^\infty(I,\lambda)$ is equal to $1$.
\end{lemma}

\begin{proof}
  We note that the operator is well defined on essentially
  bounded functions by the previous lemma. Since
  $\cP(\phi)=\phi$ the spectral radius is at least $1$.
  Taking a bounded measurable function
  $\psi:\widehat\Delta_\cQ\to\RR$ we can write
  \begin{align*}
   |\cP^n(\psi)(x)|
   &= 
   \left|
     \sum_{h\in\cH_n}
     \frac{1}{DF^n(h(x))} \; \psi (h(x))
   \right|
   \le
   \|\psi\|_\infty\cP^n(1)
   \le
   B_0\cdot\|\psi\|_\infty 
  \end{align*}
  for Lebesgue almost every $x\in I$ and every $n\ge1$, using
  the relation~\eqref{eq:bdd-unif-Delta}. Hence the spectral radius verifies
  $\text{sp}(\cP)=\limsup\sqrt[n]{ \|\cP^n\| }\le\lim\sqrt[n]{B_0}=1$.
\end{proof}

\begin{lemma}
  \label{le:RPF-conv}
  The sequence $\big(\frac1{\phi}\cP^n(\phi\cdot\psi\circ
  \widehat F^n_t)\big)_{n\ge1}$ is uniformly convergent in
  $(x,t)\in \Delta\times I$ for each continuous function
  $\psi:\widehat\Delta_\cQ\to\RR$ with compact support.
\end{lemma}

\begin{proof}
  We show that the sequence is a uniform Cauchy
  sequence. Let us fix $\epsilon>0$, then $\psi$ as in the
  statement and take $n_0$ as in \eqref{eq:unif-cont} and
  $n+k>n\ge n_0$. If we fix $(x,t)\in\Delta\times I$, then
  \begin{align*}
    \sum_{h\in\cH_n} 
    \frac{\phi\cdot\psi\circ \widehat F^n_t}{DF^n} \circ h
    &-\hspace{-.3cm}
    \sum_{h\in\cH_{n+k}} \frac{\phi\cdot\psi\circ \widehat
      F^{n+k}_t}{DF^{n+k}} \circ h
    =\hspace{-.3cm}
    \sum_{h\in\cH_n}
    \left(
      \frac{\phi\cdot\psi\circ \widehat F^n_t}{DF^n}
      -\hspace{-.2cm}
      \sum_{\ell\in\cH_{k}}
      \frac{\phi\cdot\psi\circ \widehat
        F^{n+k}_t}{DF^{n+k}} \circ \ell
    \right)\circ h
    \\
    &=
    \sum_{h\in\cH_n}
    \left(
      \frac{\phi\cdot\psi\circ \widehat F^n_t}{DF^n}
      -\hspace{-.2cm}
      \sum_{\ell\in\cH_{k}}
      \frac{\phi\cdot\psi\circ \widehat
        F^{n+k}_t}{DF^{n}\circ F^k \cdot DF^k} \circ \ell
    \right)\circ h
    \\
    &=
    \sum_{h\in\cH_n}
    \frac1{DF^n}\left(
      \phi\cdot\psi\circ \widehat F^n_t
      -\hspace{-.2cm}
      \sum_{\ell\in\cH_{k}}
      \frac{\phi\cdot\psi\circ \widehat
        F^{n+k}_t}{DF^{k}} \circ \ell
    \right)\circ h.
  \end{align*}
  Defining $\Delta_\ell^n:=\psi\circ\widehat
  F^n_t-\psi\circ\widehat F^{n+k}_t\circ \ell$ we can rewrite the
  above as
  \begin{align*}
    \sum_{h\in\cH_n}
    &\frac1{DF^n}\left(
      \phi\cdot\psi\circ \widehat F^n_t
      -\hspace{-.2cm}
      \sum_{\ell\in\cH_{k}}
      \frac{\phi\cdot\psi\circ \widehat
        F^{n}_t}{DF^{k}} \circ \ell
      +\hspace{-.2cm}
      \sum_{\ell\in\cH_{k}}
      \frac{\phi}{DF^{k}} \circ \ell \cdot\Delta_\ell^n
    \right)\circ h
    \\
    &=
    \sum_{h\in\cH_n}
    \frac1{DF^n}\left[
      \psi\circ \widehat F^n_t
      \underbrace{
      \left(\phi-
      \sum_{\ell\in\cH_{k}}
      \frac{\phi}{DF^{k}} \circ \ell
      \right)}_{\phi-\cP(\phi)=0}
      +
      \sum_{\ell\in\cH_{k}}
            \frac{\phi}{DF^{k}} \circ \ell \cdot\Delta_\ell^n
    \right]\circ h
    \\
    &=
    \sum_{h\in\cH_n}
    \frac1{DF^n}\circ h\cdot\left(
      \sum_{\ell\in\cH_{k}}
      \frac{\phi}{DF^{k}} \circ \ell \cdot\Delta_\ell^n\right)
    \circ h
    =
    \sum_{h\in\cH_{n+k}}
      \left(\frac{\phi}{DF^{n+k}}\circ h \right)
      \cdot\left(\Delta_\ell^n\circ F^k\circ h\right)
  \end{align*}
  Using $\phi >0$ and \eqref{eq:unif-cont} it follows that
  absolute value of the last expression is bounded by
  \begin{align*}
    \sum_{h\in\cH_{n+k}}
      \left|\Delta_\ell^n\circ F^k\circ h\right|
      \cdot\left(\frac{\phi}{DF^{n+k}}\circ h \right)
      \le
      \epsilon
      \sum_{h\in\cH_{n+k}}
      \frac{\phi}{DF^{n+k}}\circ h
      =
      \epsilon\phi.
  \end{align*}
  Finally, since $\epsilon>0$ and $(x,t)\in \Delta\times I$
  were arbitrarily chosen the proof is
  complete.
\end{proof}

\begin{proof}[Proof of Proposition~\ref{pr:cLdisint}] We fix
  $n=1$ for definiteness since the general case of $n>1$ is
  completely analogous.  From Lemma~\ref{le:abs-conv} the
  series $\frac1{\phi}\cP^n(\phi\cdot\psi\circ \widehat F^n_t)$ is 
  absolutely convergent.  Hence we can
  exchange the integral and the summation and apply a change
  of variables
    \begin{align*}
    \int  \int\psi\,d(\cL\upsilon)_x (t)\,d\nu(x) 
    & =
    \int_\Delta \int\psi\circ \widehat F \, d\upsilon_x \,d\nu
     =
        \sum_{h\in\cH_1}\int_{h(\Delta)}\int 
    \psi\circ \widehat F_t(x)
    \, d(\upsilon_x)(t) \, d\nu(x)
        \\
    &=
    \sum_{h\in\cH_1}\int_{h(\Delta)}
    \int 
    \phi(x)\cdot\psi\circ \widehat F_t(x)
    \, d\lambda(t) \, d\lambda(x)
    \\
        &=
    \sum_{h\in\cH_1}\int_{h(\Delta)}
    \left(\int 
    \frac{\phi\cdot\psi\circ \widehat F_t}{\phi\circ
      F\cdot DF} (h(F(x))) \, d\lambda(t)
    \right) DF(x)
    \phi(F(x))\, d\lambda(x)
    \\
        &=
    \sum_{h\in\cH_1}\int_{F(h(\Delta))}
    \left(\int 
    \frac{\phi\cdot\psi\circ \widehat F_t}{\phi\circ
      F\cdot DF} (h(x)) \, d\lambda(t)
    \right)
    \phi(x)\, d\lambda(x)
\\
    & =
    \int \sum_{h\in\cH_1}
    \int \frac{\phi\cdot\psi\circ \widehat F_t}{\phi\circ
      F\cdot DF} \circ h \, d\lambda(t)\,d\nu(x)
 \end{align*}
%
 and the statement of the lemma follows from the uniqueness
 of the disintegration.
\end{proof}

At this point we note that
$L(\psi)_x:=\lim_{n\to+\infty}\int\psi\,d(\cL^n\upsilon)_x$
is clearly a normalized, positive and bounded linear functional on
$C^0_0(\Delta_\cQ,\RR)$, thus there exists a probability
measure $\widetilde\omega_x$ such that
$L(\psi)_x=\int\psi\,d\widetilde\omega_x$. Hence, since we
have uniform convergence in Lemma~\ref{le:RPF-conv}
\begin{align*}
  \lim_{n\to+\infty}\int\int\psi\,d(\cL^n\upsilon)_x\,d\nu(x)
  =
  \int\lim_{n\to+\infty}\int\psi\,d(\cL^n\upsilon)_x\,d\nu(x)
  =
  \int\int \psi\,d\widetilde\omega_x\,d\nu(x)
\end{align*}
and because from Lemma~\ref{le.fconvergence} we also have
\begin{align*}
  \lim_{n\to+\infty}\int\int\psi\,d(\cL^n\upsilon)_x\,d\nu(x)
  =
  \int\int\psi\,d\bar\omega_x\,d\nu(x)
\end{align*}
for each continuous function with compact support, the
uniqueness of disintegration ensures that
$\bar\omega_x=\widetilde\omega_x=\lim_{n\to+\infty}(\cL^n\upsilon)_x$
for $\nu$-almost every point $x$.
Since we need to establish the smoothness of the disintegration we
\emph{define} 
$$
\bar\omega_x:=\lim_{n\to+\infty}(\cL^n\upsilon)_x
	\quad \text{ \emph{for all $x\in \Delta$}}
$$
and study in more detail this limit process.

\subsubsection{Disintegration is smooth}
\label{sec:disint-smooth}




We fix $u\in C^1_0(\Delta_\cQ,\RR)$. From
Proposition~\ref{pr:cLdisint} we have
\begin{align*}
  \bar u(x)=
  \int u(x,t)\,d\bar\omega_{x}(t) & = \int \lim_{n\to+\infty}
  \frac1{\phi} \cP^n( \phi \cdot u \circ \widehat F_t^n)(x) \;
  d\lambda(t).
\end{align*}
We shall prove that there is a well defined limit for the
expression of the derivative $ D\left[ \frac1{\phi} \cP^n(
  \phi \cdot u \circ \widehat F_t^n)\right]$ as $n\to\infty$
with uniform bounds independently of $t$.
In fact, by a straightforward computation using the chain rule and 
$|h'(x)|=|DF^n(h(x))|^{-1}$ for every $h\in\cH_n$ we get, for every $x\in \Delta$
\begin{align}
  D\left[ \frac1{\phi(x)} \cP^n( \phi \cdot u \circ \widehat
    F_t^n)(x) \right] 
  & = 
  -\frac{D\phi(x)}{\phi(x)^2} \cP^n( \phi \cdot u \circ \widehat F_t^n)(x)
     \label{eq.uc1}
     \\
     & - \frac1{\phi(x)} \sum_{h\in\cH_n}
     \left(\frac{D^2F^n}{(DF^n)^2}\cdot \frac{\phi\cdot
       u\circ\widehat F^n_t}{ DF^n}\right)\circ h(x)
     \label{eq.uc2}
     \\
     & + \frac1{\phi(x)} \sum_{h\in\cH_n} \left(\frac1{DF^n} \cdot
     \frac{D(\phi\cdot u\circ\widehat F^n_t) }{ DF^n}\right)\circ
     h(x).
     \label{eq.uc3}
\end{align}
Since $|DF^n|\geq \sigma^{-n}=(\sqrt 2)^n$ and $\cP$ is a
positive operator with spectral radius equal to one from
Lemma~\ref{le:spectral1}, we see that the absolute value of
\eqref{eq.uc3} is bounded from above by $ \frac1{\inf \phi}
\sigma^n \|\cP^n(1)\|_\infty \; \|\phi\cdot u\|_{C^1}$, and so
converges to zero.

Moreover, using the consequence \eqref{eq:conseq-renyi} of
the Renyi condition from
Subsection~\ref{sec:renyi-condit-c1}, we deduce that the
absolute value of \eqref{eq.uc2} is bounded from above by $B
n \sigma^{n-1} \frac1{\phi(x)} \cP^n( \phi \cdot u \circ
\widehat F_t^n)(x)$ and so, using Lemma~\ref{le:RPF-conv},
it converges uniformly to zero as $n\to\infty$. Hence, we get that
\begin{align*}
   D\bar u(x) 
     =
    -\frac{D\phi(x)}{\phi(x) }\int \lim_{n\to+\infty}
    \frac1{\phi} \cP^n( \phi \cdot u \circ \widehat F_t^n)
    \; d\lambda(t)
    = -\frac{D\phi(x)}{\phi(x) } \bar u(x)
\end{align*}
and so $D\bar u$ exists, hence $\bar u$ is continuous, thus
by the last identity $D\bar u$ is also continuous.  This
proves that the disintegration $(\bar\omega_x)_x$ is smooth.

\begin{remark}
  \label{rmk:smoothdens}
  The differential equation above has a solution $\bar
  u(x)=\frac1{\phi(x)}+c(t)$. But since
  \begin{align*}
    \int u \,d\eta&= 
    \int \bar u \, d\nu =
    \int (1+c(t)\phi(x))\,d\lambda(x)=1+c(t)
  \end{align*}
 we see that $c(t)\equiv c(u)=\int u\,d\eta-1$.
\end{remark}



\def\cprime{$'$}
\bibliographystyle{abbrv}

\end{document}